\pdfoutput=1 
\documentclass[11pt,a4paper,oneside]{amsart} 

\usepackage{amsmath}
\usepackage{amsfonts}
\usepackage{amssymb}
\usepackage{amsthm}
\usepackage{array}
\usepackage{enumitem}
\usepackage{mathrsfs}

\usepackage[english]{babel}
\usepackage[utf8]{inputenc}
\usepackage{enumitem}
\usepackage[a4paper,top=2cm,bottom=2cm,left=2cm,right=2cm,bindingoffset=5mm]{geometry}

\numberwithin{equation}{section}

\theoremstyle{definition}
\newtheorem{definition}{\textbf{Definition}}[section]
\newtheorem{remark}[definition]{\textbf{Remark}}

\theoremstyle{plain}
\newtheorem{proposition}[definition]{\textbf{Proposition}}
\newtheorem{corollary}[definition]{\textbf{Corollary}}	
\newtheorem{lemma}[definition]{\textbf{Lemma}}
\newtheorem{theorem}[definition]{\textbf{Theorem}}
\theoremstyle{remark}
\theoremstyle{plain}

\newtheorem{question}[definition]{Question}

\newcommand{\RR}{\mathbb{R}}

\newcommand{\ZZ}{\mathbb{Z}}

\newcommand{\e}{\mathbf{e}}

\usepackage{hyperref}      

\title{Discrete Gaussian Distributions via Theta Functions}
\author{Daniele Agostini}
\address{
  Humboldt-Universit\"{a}t zu Berlin\\
  Institut  f\"{u}r Mathematik\\
  Unter den Linden 6, 10099, Berlin, Germany}

\email[]{daniele.agostini@math.hu-berlin.de}
\author{Carlos Am\'endola}
\address{
	Technische Universit\"{a}t  M\"{u}nchen\\
	Zentrum Mathematik\\
	 Boltzmannstra{\ss}e 3, 85748, Garching, Germany}
\email[]{carlos.amendola@tum.de}

\begin{document}

\begin{abstract}
We study a discrete analogue of the classical multivariate Gaussian distribution. It is supported on the integer lattice and is parametrized by the Riemann theta function. Over the reals, the discrete Gaussian is characterized by the property of maximizing entropy, just as its continuous counterpart. We capitalize on the theta function representation to derive statistical properties. Throughout, we exhibit strong connections to the study of abelian varieties in algebraic geometry.
\end{abstract}

\maketitle

\section{Introduction}
Based on the principle of maximum entropy, it is of great interest to find, for a given class of probability distributions, the ones that maximize entropy. For a random variable $X$, with density $f(x)$, its \textit{differential entropy} as defined by Claude Shannon is
\begin{equation}\label{Def:entropy}
H[X]= \mathbb{E}[-\log(f(X))]
\end{equation}
and it is a measure of the information contained in the distribution. It is well known (see e.g. \cite{ENTR, PROB}) that the maximum entropy distribution for densities supported over a finite set or a bounded interval in $\mathbb{R}$ is given by the corresponding uniform distribution. Since there is no such distribution supported over the whole real line $\mathbb{R}$, some constraints are needed to make sure there exists a maximum entropy distribution. If one fixes the mean and the variance, the distribution with maximum entropy is the Gaussian.

This fact motivates the definition of a \textit{discrete Gaussian distribution} as the maximum entropy probability distribution supported over the integers $\mathbb{Z}$ with fixed mean and variance. In their compilation of maximum entropy distributions, Lisman and Van Zuylen in \cite{Lis} say that it ``cannot be presented in plain terms, because there is no summation procedure for the series involved''. Kemp in \cite{KEMP} proposes a somewhat explicit form for the discrete Gaussian, and later Szablowski in \cite{SZAB} observed that the normalization constant can be parametrized in terms of the Jacobi theta functions.

In this paper, we simplify and extend part of their work to higher dimensions, working over the lattice $\ZZ^g$. Central to our approach is the Riemann theta function \cite[Section 8.5]{BL},\cite[Section II.1]{MUM1}. This is the holomorphic function 
\begin{equation}
\theta \colon \mathbb{C}^g \times \mathbb{H}_g \to \mathbb{C}, \qquad \theta(u,B)= \sum_{n\in\mathbb{Z}^g} \e\left( -\frac{1}{2} n^t B n +  n^tu  \right)
\end{equation}
where $\mathbf{e}(x) = e^{2\pi  x}$ and 
$\mathbb{H}_g$ is the  Siegel right half-space, consisting of symmetric $g\times g$ complex matrices whose real part is positive definite. Using this function, we can define a family of complex probability distributions on $\mathbb{Z}^g$, that we call \textit{complex discrete Gaussian distributions}.
\begin{definition}[Complex discrete Gaussian distribution]
  Fix $(u,B)\in \mathbb{C}^g\times \mathbb{H}_g$ such that $\theta(u,B)\ne 0$. We define the discrete Gaussian distribution with parameters $(u,B)$ as the complex-valued probability distribution on $\mathbb{Z}^g$ given by
  \begin{equation}
  p_{\theta}(n; u,B) = \frac{1}{\theta(u,B)}\e\left( -\frac{1}{2} n^tB n + n^tu \right) \qquad \text{ for all } n\in \mathbb{Z}^g.
\end{equation}
\end{definition}
Observe that when $u$ and $B$ are real (that is, when their imaginary part is zero) the quantities $\e\left(-\frac{1}{2} n^tB n + n^tu\right)$ are positive real numbers, so that $p_{\theta}(n;u,B)$ is a standard real-valued probability density function on the integer lattice $\mathbb{Z}^g$. One of our main results, proven in Section 2, is that this probability distribution is the unique one that maximizes entropy among all those with same mean and covariance matrix, so that it justifies the name discrete Gaussian. 

As a consequence, discretizing the kernel of the continuous Gaussian gives indeed a distribution maximizing entropy (for its own mean and variance). This means that our definition coincides with the one usually studied in the computer science literature \cite{AhaReg, MicReg, OdedNoah, Reg}. Our observation creates a bridge to papers in this field on the topic and in particular to applications of discrete Gaussians. In fact, they play a fundamental role in \textit{lattice-based cryptography}, with one key problem being how to sample from these distributions efficiently \cite{AgaReg, GePeV}. It is our hope that our approach can be used to shed light on this problem (e.g. through numerical evaluations of the Riemann theta function).

The Riemann theta function is a core object in the study of complex abelian varieties: these are projective varieties that have at the same time the structure of an algebraic group, and they are ubiquitous in algebraic geometry. We introduce them, their parameter spaces and related concepts in Section 3. We exhibit several connections between the statistical properties of discrete Gaussians and the geometry of abelian varieties. For instance, we show in Proposition \ref{independence} that independent discrete Gaussians correspond to products of abelian varieties. 

In Section 4, we continue exploiting the representation via the Riemann theta function to study the properties of discrete Gaussians. For example, in Proposition \ref{explicitmoments} we give explicit formulas for the characteristic function and for the higher moments. In this section we also provide numerical examples where we illustrate simple modeling with discrete Gaussians and maximum likelihood estimation for them. 

Most interestingly, we show in Proposition \ref{quasiperiodicitydg} that \textit{discrete Gaussians can be interpreted, up to translation, as points on abelian varieties}, and this is our guiding principle. Moreover, this principle extends in Proposition \ref{equivalencegaussians} to families of both abelian varieties and discrete Gaussians, thus obtaining a natural relation between the set of all discrete Gaussian distributions on $\mathbb{Z}^g$ and the universal family of abelian varieties over the moduli space $\mathcal{A}_g$ \cite[Section 8.2]{BL}.

As a consequence of this principle, every statistical function of the discrete Gaussian that is translation invariant, such as central moments and cumulants, gives a well defined function on abelian varieties. In Section 5, we use this to define statistical maps of abelian varieties into projective space. These realize abelian varieties as the moment varieties of the discrete Gaussian distribution. We go on to study the geometry of these maps and we draw statistical consequences, particularly about moment identifiability.  We show in Theorems \ref{thirdmomentellcurves} and \ref{thirdmomentabsurf} that complex discrete Gaussians on $\mathbb{Z}$ or $\mathbb{Z}^2$ with the same parameter $B$ are completely determined by the moments up to order three. Moreover, for discrete Gaussians on $\mathbb{Z}$, we show in Proposition \ref{equationcubic} how to use the cubic equation of the corresponding elliptic curve to obtain explicit formulas for the higher moments in terms of the first three.

In summary, in this paper we establish a connection between algebra and statistics that echoes and complements similar results in the field of algebraic statistics \cite{DSS, PIST}. On the one hand, finite discrete exponential families have been shown to correspond to \textit{toric varieties} \cite{DSS}. On the other hand, some continuous exponential families have been linked to what are called \textit{exponential varieties} \cite{EXPVAR}. The present paper now relates some infinite discrete exponential families to abelian varieties.

\section{The discrete Gaussian distribution and the Riemann theta function} 

We start by extending the definition of the univariate discrete Gaussian in $\ZZ$ to a multivariate discrete Gaussian in $\ZZ^g$. The (continuous) multivariate Gaussian distribution given by
\begin{equation}
\label{eq:gaussian}
f(x)=\frac{1}{\sqrt{ \det(2 \pi
    \Sigma)}}e^{-\frac{1}{2} (x-\mu)^t \Sigma^{-1} (x-\mu)},
\end{equation}
is the maximum entropy distribution supported over $\RR^g$ where the mean vector $\mu \in \mathbb{R}^g$ and the symmetric positive definite $g \times g$ covariance matrix $\Sigma$ are fixed. It is thus natural to attempt the following definition.

\begin{definition}[Discrete Gaussian distribution]
The $g$-dimensional discrete Gaussian distribution is the maximum entropy probability distribution supported over $\mathbb{Z}^g$ with fixed mean vector $\mu \in \RR^g$ and covariance matrix $\Sigma \succ 0$.
\end{definition}

A priori, such a distribution may not exist nor be unique. Our main result in this section is that existence and uniqueness do hold.  

To obtain this distribution we will use the classical Riemann theta function; we recall now its definition. For any positive integer $g$ we denote by $\mathbb{H}_g$ the Siegel right half-space, which consists of symmetric $g\times g$ complex matrices whose real part is positive definite: $$\mathbb{H}_g = \{ B \in \mathbb{C}^{g\times g} \,|\, B^t = B,\, \operatorname{Re} B \succ 0 \}.$$ Then the \textit{Riemann theta function} is defined as
\begin{equation}
\theta \colon \mathbb{C}^g \times \mathbb{H}_g \to \mathbb{C}, \qquad \theta(u,B)= \sum_{n\in\mathbb{Z}^g} \e\left( -\frac{1}{2} n^t B n +  n^tu  \right)
\end{equation}
where $\mathbf{e}(x) = e^{2\pi x}$. 

Observe that the series is convergent, thanks to the condition $\operatorname{Re} B \succ 0$: in particular this shows that the function is holomorphic. We denote the zero locus of this function by $\Theta_g = \{ (u,B) \in \mathbb{C}^g\times \mathbb{H}_g \,|\, \theta(u,B) = 0 \}$; this is sometimes called the \textit{universal theta divisor}. 

\begin{remark}
	In order to emphasize the connection between geometry and statistics, we use a different notation than the usual one for the Riemann theta function. Classically, the theta function has parameters $z\in \mathbb{C}^g$ and $\tau\in \mathcal{H}_g$, where $\mathcal{H}_g$ is the Siegel upper half space of $g\times g$ symmetric complex matrices with positive definite imaginary part. To pass from our notation to the classical one, one just needs to take $z=-iu ,\tau = iB$.
\end{remark}

Now, we define the probability distributions we are going to work with.

\begin{definition}[Complex Discrete Gaussian distribution]
  For $(u,B)\in ( \mathbb{C}^g\times \mathbb{H}_g ) \setminus \Theta_g$, we define the discrete Gaussian distribution with parameters $(u,B)$ as the complex-valued probability distribution on $\mathbb{Z}^g$ given by
  \begin{equation}
  p_{\theta}(n; u,B) = \frac{1}{\theta(u,B)}\e\left( -\frac{1}{2} n^tB n + n^tu \right) \qquad \text{ for all } n\in \mathbb{Z}^g.
\end{equation}
We will denote by $X_{(u,B)}$ a random variable with values on $\mathbb{Z}^g$ and distribution given by $p_{\theta}(n;u,B)$. 
\end{definition}

We also give a name to the set of all complex discrete Gaussian distributions. We will see in Remark \ref{remarkequivalence} that this set has a natural structure of a complex manifold.

\begin{definition}[The set $\mathcal{G}_g$]\label{setdgaussians}
For a fixed $g \geq 1$, we denote by $\mathcal{G}_g$ the set of all complex discrete Gaussian distributions on $\mathbb{Z}^g$.
\end{definition}

Observe that when $u$ and $B$ are real (that is, when their imaginary part is zero) the quantities $\e\left(-\frac{1}{2} n^tB n + n^tu\right)$ are positive real numbers, so that $p_{\theta}(n;u,B)$ is a standard real-valued probability density function on $\mathbb{Z}^g$. We will show that this probability distribution is the unique one that maximizes entropy among all those with same mean and covariance matrix. In particular, the two definitions of discrete Gaussian agree in this case.

\begin{theorem}\label{mainthm}
Fix a vector $\mu \in \mathbb{R}^g$ and a positive definite symmetric matrix $\Sigma \in \operatorname{Sym}^2\mathbb{R}^{g}$. Then there exists a unique distribution supported on $\ZZ^g$ with mean vector $\mu$ and covariance matrix $\Sigma$ that maximizes entropy. This distribution is given by

\begin{equation} \label{discGauss}
p_{\theta}(n;u,B) = \frac{1}{\theta(u,B)}\e\left( -\frac{1}{2}n^tB n + n^tu \right) 
\end{equation}
for some unique $u \in \mathbb{R}^g$ and $B \in \mathbb{H}^g$ real.
\end{theorem}

It is natural to ask how one can compute effectively such $u,B$ from given $\mu,\Sigma$. See Remark \ref{MLE} and Subsection \ref{rmk:comp}. For a explicit expression of the maximized entropy see Proposition \ref{prop:entropy}.
 
\begin{remark}
When $B$ has imaginary part zero, it is a positive definite real matrix. Then the function $n \mapsto -2\pi  \left( \frac{1}{2} n^tB n \right)$ is a negative definite quadratic form. Thus $p_\theta$ is a log-concave density (looking at its piecewise linear extension in $\mathbb{R}^g$), just as its continuous counterpart.  
\end{remark}

\begin{remark}
  Set $g=1$ and take two real parameters $u\in\mathbb{R}$ and $B \in \mathbb{H}_1 \cap \mathbb{R} = \mathbb{R}^+$. Then if we set
  $q = \e\left(  -B \right), \lambda = \e\left( -\frac{1}{2} B + u \right)$
we can rewrite
\begin{equation}
 \e\left(-\frac{1}{2}n^2B + nu\right) = \lambda^n q^{\frac{n(n-1)}{2}}.
\end{equation}
Therefore, our one-dimensional discrete Gaussian coincides with the one defined by Kemp in \cite{KEMP}.
\end{remark}

We see now that the density we are proposing is very special in the statistical sense. Indeed, it belongs to an \textit{exponential family}. 

Exponential families are of fundamental importance in statistics \cite{BROWN} and we briefly recall their definition here. Many common examples of distributions belong to an exponential family. An exponential family on a space $\mathcal{X}$ has density 
\begin{equation} \label{expfam}
p(\eta,x) = h(x) e^{\left\langle \eta, T(x) \right\rangle - A(\eta)} \qquad \text{ for } x\in \mathcal{X}
\end{equation}
where $\eta \in \mathbb{R}^n$ is a parameter, $T\colon \mathcal{X} \to \mathbb{R}^n$ is a measurable function, and $\langle \cdot , \cdot \rangle$ is the standard scalar product on $\mathbb{R}^n$.
In this case, $h$ is known as the \textit{base measure}, $T$ as the \textit{sufficient statistics} and $A$ as the \textit{log-partition function}. The space of canonical parameters is $\{ \eta \in \RR^n \vert A(\eta) < \infty \}$ and it is convex. The exponential family is \textit{regular} if the space of canonical parameters is a nonempty open set and it is \textit{minimal} if the image of $T$ does not lie in a proper affine subspace of $\RR^n$.

\begin{proposition} \label{thm:expfam}
Fix real parameters $u,B$. The density function $p_{\theta}(n,u,B)$ over $\ZZ^g$ in Theorem \ref{mainthm} belongs to a minimal regular exponential family of distributions. 
\end{proposition}
\begin{proof}
Indeed, we can rewrite $p_{\theta}(n,u,B)$ in the exponential family form:
\begin{equation}
p_{\theta}(n,u,B) = \dfrac{e^{2 \pi  \left(  n^tu - \frac{1}{2} n^t B n \right) }}{ \theta(u,B)} = e^{ \left\langle (u,B) , 2\pi (n, -\frac{1}{2} nn^t )\right\rangle - \log(\theta(u,B))},
\end{equation}
so that we have identity base measure $h(n)=1$, sufficient statistics $T(n)=2\pi (n,- \frac{1}{2}nn^t )$ and log-partition function $A(u,B)=\log(\theta(u,B))$.
\end{proof}

\begin{corollary} \label{cor:bij}
There is a bijection between the set of real canonical parameters $(u,B) \in \mathbb{R}^g \times \mathbb{H}_g$ and the moments $(\mu,\Sigma) \in \RR^g \times \operatorname{Sym}^2 \mathbb{R}^g$ with $\Sigma \succ 0$. It is given up to scaling by the gradient of the log-partition function $\nabla A(u,B) = \dfrac{1}{\theta(u,B)}\nabla \theta(u,B)$. 
\end{corollary}
\begin{proof}
From classical theory of minimal regular exponential families \cite{BROWN,EXPVAR}, there is a bijection between the space of canonical parameters $\{ \eta \,\vert\, A(\eta)<\infty \}$ and the space of sufficient statistics $\operatorname{conv}(T(x) \vert x\in \mathcal{X})$, and it is given by the gradient of the log-partition function $\nabla A(\eta)$. 
\end{proof}

\begin{remark} \label{MLE}
Furthermore, the inverse map $(\mu,\Sigma) \rightarrow (u,B)$ is known to give the maximum likelihood estimate (MLE) in the following sense. Given a sample $x$ from a distribution in the exponential family, we can compute its sufficient statistics $t=T(x)$ (in our case multiples of the sample mean $\mu$ and sample covariance $\Sigma$).  For the given sufficient statistics $t$, solving the equation $\nabla A(u,B) =  t $ for $u,B$ is equivalent to solving the likelihood equations. See Subsection \ref{rmk:comp} for an example.
\end{remark}

Finally, exponential families are precisely the right candidates to be maximum entropy distributions \cite{ENTR}. We present the argument for this in our case.

\begin{corollary} \label{cor:maxent}
Fix real parameters $u,B$ and let $\mu,\Sigma$ be the moment vector and the covariance matrix of the corresponding discrete Gaussian. Among all densities $q(n)$ of support $\mathbb{Z}^g$ with fixed mean vector $\mu$ and covariance matrix $\Sigma$, the distribution $p_\theta(n)=p(n,u,B)$ maximizes entropy. 
\end{corollary}
\begin{proof}
We observe that $q$ having the first two matching moments to $p_\theta$ means that when we sum $q(n)$ against the linear form $n^t u $ and against the quadratic form $n^t B n$, it is equivalent to doing it with $p_\theta(n)$. That is,  
$$\sum_{n \in \mathbb{Z}^g} q(n) n^t u  = \sum_{n \in \mathbb{Z}^g} p(n) n^tu \quad , \quad  \sum_{n \in \mathbb{Z}^g} q(n) n^t B n= \sum_{n \in \mathbb{Z}^g} p(n) n^t B n. $$
Thus,
$$\sum_{n \in \mathbb{Z}^g} q(n) \log(p_\theta(n))  = \sum_{n \in \mathbb{Z}^g} p(n) \log(p_\theta(n)). $$
With this in mind, we proceed:
\begin{align*}
H(q) &= -\sum_{n \in \mathbb{Z}^g} q(n) \log(q(n))\\ 
        &=  -\sum_{n \in \mathbb{Z}^g} q(n) \log \dfrac{q(n)}{p_\theta(n)} -\sum_{n \in \mathbb{Z}^g} q(n) \log(p_\theta(n))\\
        &= - KL(q\vert p_\theta) - \sum_{n \in \mathbb{Z}^g} p(n) \log(p_\theta(n)) \\
        &\leq - \sum_{n \in \mathbb{Z}^g} p(n) \log(p_\theta(n)) \\
        &= H(p_\theta)
\end{align*}
where $KL(q\vert p_\theta)$ is the Kullback-Leibler divergence (always non-negative by the classical Jensen's inequality). Further, equality holds if and only if $q=p_\theta$ almost everywhere. Since we are working with the counting measure over the integer lattice, we have uniqueness.
\end{proof}

Combining Corollary \ref{cor:bij} and Corollary \ref{cor:maxent} we obtain the result stated in Theorem \ref{mainthm}.

\section{Abelian varieties and discrete Gaussians}

The Riemann theta function is a central object in the study of complex abelian varieties. These are projective varieties that have at the same time the structure of an algebraic group, and they are fundamental objects in algebraic geometry. Especially important are principally polarized abelian varieties: these are pairs $(A,\Theta)$, where $A$ is an abelian variety and $\Theta$ is an ample divisor on $A$ such that $h^0(A,\Theta)=1$.

The theta function can be used to build a universal family of these varieties. This is a very classical construction that we recall briefly here. For more details one can look into \cite[Chapter 8]{BL}.

We should however point out a slight difference between our construction and the classical one: for some sources, for example \cite{BL}, an isomorphism of principally polarized abelian varieties $(A,\Theta)$ and $(A',\Theta')$ is considered to be an isomorphism of groups $F\colon A \to A'$ such that $F^{-1}(\Theta')$ and $\Theta$ differ by a translation. However, for our purposes we will need to fix the theta divisors, so that, we will define an isomorphism between $(A,\Theta)$ and $(A',\Theta')$ to be an isomorphism of varieties $F\colon A \to A'$ such that $F^{-1}(\Theta')=\Theta$.

\subsection{Parameter spaces of abelian varieties}\label{parameterabvar}
Fix an integer $g\geq 1$, and for any $B \in \mathbb{H}_g$ consider the subgroup $\Lambda_B = \{ im + B n \,|\, m,n\in\mathbb{Z}^g \} \subseteq \mathbb{C}^g$. Since $B \in \mathbb{H}_g$ one can see that $\Lambda_{B}$ is a lattice, so that $A_{B} = \mathbb{C}^g/\Lambda_{B}$ is a complex torus. Moreover, the theta function is quasiperiodic with respect to this lattice, meaning that for all $n,m \in \mathbb{Z}^g$ and $u \in \mathbb{C}^g,B \in \mathbb{H}_g$ we have:
\begin{equation}\label{quasiperiodicitytheta}
\theta(u+im+B n, B) = \e\left(\frac{1}{2} n^t B n + n^t u\right) \theta(u,B).
\end{equation}
In particular, the theta divisor $\Theta_{B}:=\{ u \in \mathbb{C}^g \,|\, \theta(u,B) = 0 \}$ is invariant under $\Lambda_B$, so that it descends to a divisor on $A_{B}$ that we denote by $\Theta_{B}$ again.

Riemann proved that $(A_B,\Theta_B)$ is a principally polarized abelian variety, and moreover he also showed that every principally polarized abelian variety of dimension $g$ is isomorphic to one of these. More precisely, for any principally polarized abelian variety $(A,\Theta)$, there is a certain $B\in\mathbb{H}_g$ and an isomorphism $F\colon A \to A_{B}$ such that $F^{-1}(\Theta_{B}) = \Theta$. Hence, with this interpretation, the space $\mathbb{H}_g$ becomes a parameter space for all principally polarized abelian varieties.

One can go further and construct a moduli space of these varieties. The space $\mathbb{H}_g$ has a natural action of the symplectic group
\begin{equation}
Sp(2g,\mathbb{Z}) = \left\{ M \in M(2g,\mathbb{Z}) \, \bigg| \, M^t \begin{pmatrix} 0 & I_g \\ -I_g & 0 \end{pmatrix}M = \begin{pmatrix} 0 & I_g \\ -I_g & 0 \end{pmatrix}    \right\}
\end{equation}
defined  as follows:
\begin{equation}
M = \begin{pmatrix} \alpha & \beta \\ \gamma & \delta \end{pmatrix} \in Sp(2g,\mathbb{Z}),\,\, B \in \mathbb{H}_g \qquad MB := -i(\alpha B - i\beta)(\gamma B-i\delta)^{-1}. 
\end{equation}
For $M$ and $B$ as above, we can define an invertible $\mathbb{C}$-linear map
\begin{equation}
\widetilde{f}_{M,B} \colon \mathbb{C}^g \to \mathbb{C}^g, \qquad u \mapsto  -i(\gamma B -i \delta)^{-t}u 
\end{equation}
which in turn induces an isomorphism of abelian varieties $f_{M,B} \colon A_{B} \to A_{MB}$. However, this is in general not an isomorphism of polarized abelian varieties, since the pullback $f_{M,B}^{-1}(\Theta_{MB})$ could differ from $\Theta_{B}$ by a translation. To fix this, one defines (see \cite[Lemma 8.4.1]{BL})
\begin{equation}
c_1 = \frac{1}{2}\operatorname{diag}(\gamma \delta^t), \qquad c_2 = \frac{1}{2}\operatorname{diag}(\alpha \beta^t), \qquad  c_{M,B} = (MB) c_1 + ic_2 
\end{equation}
and then considers the affine transformation
\begin{equation}\label{isomorphismMtau}
\widetilde{F}_{M,B}\colon \mathbb{C}^g \to \mathbb{C}^g, \qquad \widetilde{F}_{M,B}(u) = \widetilde{f}_{M,B}(u) + c_{M,B} 
\end{equation}
which in turn induces a map $F_{M,B}\colon A_{B} \to A_{MB}$. One can see that this is again an isomorphism and moreover $F_{M,B}^{-1}(\Theta_{MB}) = \Theta_{B}$. Hence, this shows that the two polarized abelian varieties $(A_B,\Theta_B)$ and $(A_{MB},\Theta_{MB})$ are isomorphic. Moreover, it turns out that, up to translations, any isomorphism between two polarized abelian varieties $(A_B,\Theta_{B})$ and $(A_{B'},\Theta_{B'})$ is of this form. 

Hence, the quotient $\mathcal{A}_g = \mathbb{H}_g/Sp(2g,\mathbb{Z})$ is a natural parameter space for isomorphism classes of  principally polarized abelian varieties: it is usually called the moduli space of abelian varieties of dimension $g$, and it is an algebraic variety itself \cite[Remark 8.10.4]{BL}.

\subsection{Universal families of abelian varieties}\label{familiesabvar}

Out of the previous discussion we also can actually construct universal families of abelian varieties. The group $\mathbb{Z}^g\oplus \mathbb{Z}^g$ acts on $\mathbb{C}^g \times \mathbb{H}_g$  by
\begin{equation} (m,n) \cdot (u,B) = (u+im+B n, B).\end{equation}
The quotient $\mathcal{X}_g =  (\mathbb{C}^g\times \mathbb{H}_g)/(\mathbb{Z}^g\oplus\mathbb{Z}^g)$ is a complex manifold equipped with a map $\mathcal{X}_g \to \mathbb{H}_g$ such that the fiber over $B$ is precisely the abelian variety $A_{B}$.  Moreover, by quasiperiodicity (\ref{quasiperiodicitytheta}), the universal theta divisor $\Theta_g =\{(u,B) \in \mathbb{C}^g\times \mathbb{H}_g \,|\, \theta(u,B)=0 \}$ passes to the quotient $\mathcal{X}_g$ and  defines another divisor in $\mathcal{X}_g$ whose restriction to $A_{B}$ is precisely the theta divisor $\Theta_{B}$. Hence, the map $\mathcal{X}_g \to \mathbb{H}_g$, together with $\Theta_g$, can be considered as the universal family of abelian varieties over $\mathbb{H}_g$.

We can do something similar with the moduli space. Indeed, the action of $Sp(2g,\mathbb{Z})$ on $\mathbb{H}_g$, extends (see \cite[Lemma 8.8.1]{BL}) to an action of the group $G_g = (\mathbb{Z}^g\oplus \mathbb{Z}^g)\rtimes Sp(2g,\mathbb{Z})$ on $\mathbb{C}^g\times \mathbb{H}_g$, given by
\begin{equation}\label{actionG}
((m,n),M) \cdot (u,B) = (\widetilde{F}_{M,B}(u)+im +(MB)n ,MB).
\end{equation}
The quotient $\mathcal{U}_g = (\mathbb{C}^g\times \mathbb{H}_g)/G_g$ by this action can also be seen as a quotient $\mathcal{X}_g/Sp(2g,\mathbb{Z})$. In particular, it has a natural map $\mathcal{U}_g\to \mathcal{A}_g$. Moreover, the Theta Transformation Formula \cite[Theorem 8.6.1]{BL} tells us explicitly how the theta function changes under $G_g$: for each $M = \left(\begin{smallmatrix} \alpha & \beta \\ \gamma & \delta \end{smallmatrix}\right)\in Sp(2g,\mathbb{Z})$ and $u\in \mathbb{C}^g,B \in \mathbb{H}_g$ we have 
\begin{equation}\label{thetatransformation}
\theta(F_{M}(u,B),MB) = C(M,B,u)
\theta(u,B)
\end{equation}
for a certain explicit function $C(M,B,u)$ which never vanishes. In particular, this tells us that the universal theta divisor $\Theta_g$ is invariant under the action of $G_g$, so that it passes to the quotient $\mathcal{U}_g$. In this setting, the map $\mathcal{U}_g\to \mathcal{A}_g$, together with $\Theta_g$ is sometimes called the universal family over the moduli space $\mathcal{A}_g$.

\subsection{Discrete Gaussian distributions and abelian varieties}
In the following, we study further statistical properties of discrete Gaussians, in light of the connection to abelian varieties.

The key observation is that the quasiperiodicity property has a transparent interpretation in terms of the discrete Gaussian. Recall that if $(u,B) \in \mathbb{C}^g\times \mathbb{H}_g \setminus \Theta_g$, we denote by $X_{(u,B)}$ a random variable with discrete Gaussian distribution of parameters $(u,B)$. We will use the standard notation $X\sim Y$ to denote that two random variables have the same distribution.

\begin{proposition}[Quasiperiodicity]\label{quasiperiodicitydg}
Let $(u,B)\in \mathbb{C}^g \times \mathbb{H}_g \setminus \Theta_g$. Then for every $m,n \in \mathbb{Z}^g$ we have that
\begin{equation}
X_{(u+im+B n,B)} \sim X_{(u,B)} +n.
\end{equation}
\end{proposition}
\begin{proof}
First we observe that by the quasiperiodicity of the theta function (\ref{quasiperiodicitytheta}) we have that $(u+im+B n,B) \notin \Theta_g$, so that it makes sense to speak of the variable $X_{(u+im+B n,B)}$. Now, we fix more generally $(u,B)\in \mathbb{C}^g\times \mathbb{H}_g$ and $m,n\in \mathbb{Z}^g$ and for all $h\in\mathbb{Z}^g$  we can compute 
\begin{align*}
\e&\left(-\frac{1}{2} h^tB h + h^t (u + im + B n)\right)  = \e\left(-\frac{1}{2}h^tB h + h^t u + h^tB n \right) \\
 & = \e\left(-\frac{1}{2}(h-n)^tB (h-n) + (h-n)^t u -\frac{1}{2}n^tB h - \frac{1}{2} h^tB n 
 + \frac{1}{2}n^t B n + n^t u +  h^tB n \right) \\
 & = \e \left( \frac{1}{2}n^tB n + n^t u \right) \e\left(-\frac{1}{2}(h-n)^tB (h-n) + (h-n)^t u\right) \e\left( - \frac{1}{2} h^tB n - \frac{1}{2}n^t B h  +  h^tB n \right).
\end{align*}
However, since $B$ is symmetric we see that the last factor in this expression is $1$. Notice that this shows in particular that the theta function is quasiperiodic as in (\ref{quasiperiodicitytheta}). Now we can compute 
\begin{align*}
\mathbb{P}(X_{(u+im+B n,B)} = h) &= \frac{\e\left(-\frac{1}{2}h^tB h + h^t (u + im + B n)\right)}{\theta(u+im + B n , B)} \\ 
                                      &= \frac{\e\left(\frac{1}{2}n^tB n + n^t u\right)\e\left(-\frac{1}{2}(h-n)^tB (h-n) + (h-n)^t u\right)}{\e\left(\frac{1}{2}n^tB n + n^t u\right) \theta(u,B)} \\
  & = \frac{\e\left(-\frac{1}{2}(h-n)^tB (h-n) + (h-n)^t u\right)}{ \theta(u,B)} = \mathbb{P}(X_{(u,B)} = h-n)
\end{align*}
which is precisely what we want.
\end{proof}

\begin{remark}\label{rmk:dream}
This result tells us that if we fix a parameter $B\in\mathbb{H}_g$, then discrete Gaussian distributions $X_{(u,B)}$ correspond, up to translation, to points in the open subset $A_{B}\setminus \Theta_{B} \subseteq A_{B}$. 
\end{remark}

Hence, we see a direct connection between discrete Gaussians and abelian varieties. More precisely, the next Proposition \ref{equivalencegaussians} shows that we can relate the set $\mathcal{G}_g$ of all discrete Gaussian distributions to the universal family $\mathcal{U}_g$.  

\begin{proposition}[Equivalence of discrete Gaussians]\label{equivalencegaussians}
Let $(u,B),(u',B') \in \mathbb{C}^g\times \mathbb{H}_g \setminus \Theta_g$.  Then $X_{(u,B)} \sim X_{(u',B')}$ if and only if
\begin{equation}
u = u' + i\left( \frac{1}{2}\operatorname{diag}(\beta) + a\right), \qquad B = B' - i\beta 
\end{equation}
where $\beta\in \operatorname{Sym}^2\mathbb{Z}^{g}$ is a symmetric matrix with integer coefficients, $\operatorname{diag}(\beta)$ is the diagonal of $\beta$ and  $a\in \mathbb{Z}^g$.
\end{proposition}
\begin{proof}
  The two discrete Gaussians $X_{(u,B)}$ and $X_{(u',B')}$ have the same distribution if and only if
  \begin{equation}
  \frac{\e\left(-\frac{1}{2}n^tB n + n^t u\right)}{\theta(u,B)}  = \frac{\e\left(-\frac{1}{2}n^tB' n + n^t u'\right)}{\theta(u',B')} \qquad \text{ for all } n\in \mathbb{Z}^g.
\end{equation}
We can rewrite this as
\begin{equation}\label{equivalence2}
\e\left(-\frac{1}{2}n^t(B-B') n + n^t (u-u')\right) = \frac{\theta(u,B)}{\theta(u',B')} \qquad \text{ for all } n\in \mathbb{Z}^g.
\end{equation}
In particular, this shows that the left hand side is independent of $n$, so that plugging $n=0$ gives that 
\begin{equation}\label{equivalence3}
\e\left(-\frac{1}{2}n^t(B-B') n + n^t (u-u')\right) = 1 \qquad \text{ for all } n\in \mathbb{Z}^g.
\end{equation}
Conversely, if this last condition holds, then it is clear that $\theta(u,B) = \theta(u',B')$, so that (\ref{equivalence2}) holds as well. This proves that $X_{(u,B)}\sim X_{(u',B')}$ is equivalent to (\ref{equivalence3}). We can rewrite (\ref{equivalence3}) as  
\begin{equation}\label{equivalence4}
-\frac{1}{2}n^t(B-B') n + n^t (u-u') \in i\mathbb{Z} \qquad \text{ for all } n \in  \mathbb{Z}^g.
\end{equation}
Now, let's set $B-B'=-i\beta$ and $u-u' = \frac{1}{2}i\operatorname{diag}(\beta) + ia$: the condition becomes
\begin{equation}\label{equivalence5}
\frac{1}{2}n^t\beta n+\frac{1}{2}n^t\operatorname{diag}(\beta) + n^ta \in \mathbb{Z} \qquad \text{ for all } n\in \mathbb{Z}^g.
\end{equation}
We need to prove that this holds if and only if $a,\beta$ have integer coefficients.
Suppose first that $a,\beta$ have integer coefficients: the coordinates corresponding to off-diagonal entries are indeed integers (since $\beta$ is symmetric), so (\ref{equivalence5}) holds if
\begin{equation}
\frac{1}{2}n^2 \beta_{ii} + \frac{1}{2}n \beta_{ii} = \frac{n(n+1)}{2} \beta_{ii} \in \mathbb{Z} \qquad \text{ for all } n\in \mathbb{Z}, i \in \{ 1, \dots, g \},
\end{equation}
but this is true because $\beta_{ii} \in \mathbb{Z}$ and $n(n+1)$ is even. 

Conversely, suppose that (\ref{equivalence5}) holds and let's prove that $a,\beta$ have integer coefficients. Let $e_i$ denote the $i$-th vector of the canonical basis. Taking $n=e_i$ and $n=-e_i$ in (\ref{equivalence5})  we get that $\beta_{ii}\in \mathbb{Z}$ and $a_i\in \mathbb{Z}$ for all $i$. To conclude, we just need to check that $\beta_{ij} \in \mathbb{Z}$ for all $i\ne j$ and this follows again from (\ref{equivalence5}) taking $n=e_i+e_j$. 
\end{proof}

\begin{remark}\label{remarkequivalence}
We can interpret this result via universal families of abelian varieties as follows: the group $N_{g} = \mathbb{Z}^g \oplus \operatorname{Sym}^2 \mathbb{Z}^g$ embeds in the group $G_g = (\mathbb{Z}^g\oplus \mathbb{Z}^g)\rtimes Sp(2g,\mathbb{Z})$ via the map
\begin{equation}
N_g \hookrightarrow G_g \qquad (a,\beta) \mapsto \left( (a,0) , \begin{pmatrix} I_g & \beta \\ 0 & I_g \end{pmatrix} \right).
\end{equation}
In particular, the action (\ref{actionG}) of $G_g$ on $\mathbb{C}^g\times \mathbb{H}_g$ restricts to an action of $N_g$ as follows:
\begin{equation}
(a,\beta) \cdot (u,B) = \left(u+ia+i\frac{1}{2}\operatorname{diag}(\beta),B-i\beta\right).
\end{equation}
Then, Proposition \ref{equivalencegaussians} says precisely that the set $\mathcal{G}_g$ of all discrete Gaussian distributions is naturally identified with the quotient $((\mathbb{C}^g \times \mathbb{H}_{g}) \setminus \Theta_g) / N_g$. In particular, since the action of $N_g$ is free and properly discontinuous, we see that $\mathcal{G}_g$ has a natural structure of a complex manifold.

Moreover, since $N_g$ is a subgroup of $G_g$, we have a natural map
$\mathcal{G}_g  \longrightarrow  \mathcal{U}_g$, so that the set $\mathcal{G}_g$ of discrete Gaussians is an intermediate space between $\mathbb{C}^g\times \mathbb{H}_g$ and the universal family $\mathcal{U}_g$. 
\end{remark}

\subsection{Discrete Gaussians and affine transformations}

The class of continuous Gaussians is invariant under the action of affine automorphisms of $\mathbb{R}^g$. The same is true for discrete Gaussians and affine automorphisms of $\mathbb{Z}^g$. The case of translations is covered by Proposition \ref{quasiperiodicitydg}, so we just need to consider linear automorphisms.

\begin{proposition}\label{linearauto}
  Let $X_{(u,B)}$ be a discrete Gaussian on $\mathbb{Z}^g$. Then, for any $\alpha\in GL(g,\mathbb{Z})$ we have
  \begin{equation}
  \alpha X_{(u,B)} \sim X_{(\alpha^{-t}u,{\alpha^{-t}}B \alpha^{-1})}. 
  \end{equation}
\end{proposition}
\begin{proof}
  We just write down everything explicitly: if $n\in\mathbb{Z}^g$, then we see that
  \begin{align}
    \e\left(-\frac{1}{2} (\alpha^{-1}n)^tB (\alpha^{-1}n) + (\alpha^{-1}n)^tu \right) = \e\left(-\frac{1}{2} n^{t}(\alpha^{-t}B \alpha^{-1})n + n^t(\alpha^{-t}u) \right)
  \end{align}
  In particular, summing over $\mathbb{Z}^g$ we get that $\theta(u,B) = \theta(\alpha^{-t}u,\alpha^{-t}B \alpha^{-1})$. Putting these two together, we conclude that both distributions have the same probability mass functions.
\end{proof}

\begin{remark}
We can also interpret this result as in Remark \ref{remarkequivalence}. Indeed, the group $GL(g,\mathbb{Z})$ embeds in the group $G_g = (\mathbb{Z}^g\oplus \mathbb{Z}^g)\rtimes Sp(2g,\mathbb{Z})$ via the map
\begin{equation}
GL(g,\mathbb{Z}) \hookrightarrow G_g, \qquad \alpha \mapsto \left( (0,0) , \begin{pmatrix} \alpha^{-t} & 0 \\ 0 & \alpha \end{pmatrix} \right).
\end{equation}
In particular, the action (\ref{actionG}) of $G_g$ on $\mathbb{C}^g\times \mathbb{H}_g$ restricts to an action of $N_g$ as follows:
\begin{equation}
\alpha \cdot (u,B) = \left(\alpha^{-t}u,\alpha^{-t}B \alpha^{-1} \right),
\end{equation}
which is exactly what appears in Proposition \ref{linearauto}.
\end{remark}

\begin{remark}[Parity]\label{parity}
In particular, it follows immediately from Proposition \ref{linearauto} that $X_{(-u,B)} \sim -X_{(u,B)}$. This gives a parity property for discrete Gaussians that reflects the analogous property for the Riemann Theta function.
\end{remark}

However, in general it is not true that an arbitrary linear transformation of a discrete Gaussian is again a discrete Gaussian. Still, we can control the difference, and thanks to Proposition \ref{linearauto}, it is enough to do it for projections.

\begin{proposition}\label{projectiondiscreteGaussian}
  Let $X=(X_1,X_2)$ be a discrete Gaussian random variable on $\mathbb{Z}^{g_1}\times \mathbb{Z}^{g_2}$ with parameters $u=(u_1,u_2)$ and $B = \left( \begin{smallmatrix} B_{11} & B_{12} \\ B_{12}^t & B_{22} \end{smallmatrix} \right)$.

  Then $X_1$ has probability mass function:
  \begin{equation}
  \mathbb{P}(X_1 = n_1) = \frac{\theta(u_1,B_{11})\theta(u_2-B_{12}^t n_1,B_{22})}{\theta(u,B)} \mathbb{P}(X_{(u_1,B_{11})}=n_1).
\end{equation}
\end{proposition}
\begin{proof}
This is a straightforward computation:
\begin{align*}
  \mathbb{P}(X_1 = n_1) &= \sum_{n_2\in \mathbb{Z}^{g_2}} \mathbb{P}((X_1,X_2) = (n_1,n_2)) \\
                                                                                            &  = \frac{1}{\theta(u,B)}\sum_{n_2\in\mathbb{Z}^{g_2}} \e\left(- \frac{1}{2} (n_1^t,n_2^t)\begin{pmatrix} B_{11} & B_{12} \\ B_{12}^t & B_{22} \end{pmatrix}\begin{pmatrix} n_1 \\ n_2 \end{pmatrix} + (n_1^t,n_2^t)\begin{pmatrix} u_1 \\ u_2 \end{pmatrix}\right) \\
                        & = \frac{1}{\theta(u,B)} \sum_{n_2\in\mathbb{Z}^{g_2}}\e\left(- \frac{1}{2}n_1^tB_{11}n_1 - n_2^tB_{12}^t n_1 - \frac{1}{2}n_2^tB_{22}n_2 + n_1^t u_1 + n_2^t u_2  \right) \\
                        & = \frac{1}{\theta(u,B)} \e\left( -\frac{1}{2}n_1^tB_{11}n_1 + n_1^t u_1 \right)\sum_{n_2\in\mathbb{Z}^{g_2}} \e\left( -\frac{1}{2}n_2^tB_{22}n_2 + n_2^t(u_2-B_{12}^t n_1) \right) \\
  & = \frac{\theta(u_2-B_{12}^t n_1,B_{22})}{\theta(u,B)}\e\left(- \frac{1}{2}n_1^tB_{11}n_1 + n_1^t u_1 \right) \\ 
  & = \frac{\theta(u_1,B_{11})\theta(u_2-B_{12}^t n_1,B_{22})}{\theta(u,B)} \mathbb{P}(X_{(u_1,B_{11})}=n_1)
\end{align*}
\end{proof}

In particular, the components of a discrete Gaussian are not discrete Gaussian themselves, but we see now that things work well when they are independent. 

\subsection{Independence for discrete Gaussians}

We can characterize independence of joint discrete Gaussians in a way analogous to continuous Gaussians.

\begin{proposition}\label{independence}
  Let $X=(X_1,X_2)$ be a discrete random variable on $\mathbb{Z}^{g_1}\times \mathbb{Z}^{g_2}$.
  Then the following are equivalent:
  \begin{enumerate}
  \item $X$ is a discrete Gaussian with $X_1,X_2$ independent.  
  \item $X$ is a discrete Gaussian with parameter $B = \left( \begin{smallmatrix} B_{11} & 0 \\ 0 & B_{22} \end{smallmatrix} \right)$.
  \item $X_1,X_2$ are independent discrete Gaussians with parameters $B_{11}$ and $B_{22}$ respectively.
  \end{enumerate}
  Moreover, if $X$ is a discrete Gaussian with real parameters $B$ and $u$, these conditions are equivalent to
  \begin{enumerate}
  	\setcounter{enumi}{3}
  	\item $X_1,X_2$ are uncorrelated.
  \end{enumerate} 
\end{proposition}
\begin{proof}
 Let $X$ is a discrete Gaussian with parameters $u=(u_1,u_2)$ and $B = \left( \begin{smallmatrix} B_{11} & B_{12} \\ B_{12}^t & B_{22} \end{smallmatrix} \right)$. With the same computations as in Lemma \ref{projectiondiscreteGaussian} we see that
  \begin{align*}
   \mathbb{P}((X_1,X_2) = (n_1,n_2)) &=  \frac{1}{\theta(u,B)} \e\left( -\frac{1}{2}n_1^t B_{11}n_1 - n_1^t B_{12}n_2 - \frac{1}{2}n_2^tB_{22}n_2 + n_1^t u_1 + n_2^t u_2  \right), \\
    \mathbb{P}(X_1 = n_1 ) &= \frac{\theta(u_2-B_{12}^t n_1,B_{22})}{\theta(u,B)}\e\left(- \frac{1}{2}n_1^t B_{11}n_1 + n_1^t u_1 \right), \\
    \mathbb{P}(X_2 = n_2) & = \frac{\theta(u_1-B_{12}n_2,B_{11})}{\theta(u,B)}\e\left(- \frac{1}{2}n_2^t B_{22}n_2 + n_2^t u_2 \right).
  \end{align*}
  Hence $\mathbb{P}((X_1,X_2)=(n_1,n_2)) = \mathbb{P}(X_1=n_1)\mathbb{P}(X_2=n_2)$ if and only if
  \begin{equation}\label{relationindependence}
  \theta(u,B) \cdot \e\left( n_1^t B_{12}n_2 \right) = \theta(u_1-B_{12}n_2,B_{11})\theta(u_2-B_{12}^t n_1,B_{22}) 
\end{equation}
 We first show that (1) implies (2). If $X_1,X_2$ are independent, this means precisely that the relation (\ref{relationindependence}) holds for every $n_1,n_2$. Choosing $n_1=0$ shows that $\theta(u_1-B_{12}n_2,B_{11})$ is a constant function of $n_2$, and choosing $n_2=0$ shows that $\theta(u_2-B_{12}^t n_1,B_{22})$ is a constant function of $n_1$. Hence, it must be that $\e\left( n_1^t B_{12}n_2 \right)$ is independent of $n_1,n_2$, and by choosing one of the two to be zero, we see that $\e\left( n_1^t B_{12}n_2 \right)=1$. This is the same as asking that $n_1^t B_{12}n_2 \in \mathbb{Z}i$ for all $n_1,n_2$, which in turn means that $B_{12}$ has coefficients in $\mathbb{Z}i$. Then, thanks to Proposition \ref{equivalencegaussians}, we can assume $B_{12}=0$, which proves (2).

To show that (2) implies (3) observe that if we set $B_{12}=0$ in (\ref{relationindependence}), then the relation reduces to $\theta(u,B) = \theta(u_1,B_{11})\theta(u_2,B_{22})$, which is true by an explicit computation. Hence, we get that $X_1,X_2$ are independent. To show that they are discrete Gaussians, it is enough to plug $B_{12}=0$ in Proposition \ref{projectiondiscreteGaussian} and use again that $\theta(u,B) = \theta(u_1,B_{11})\theta(u_2,B_{22})$. In the same way, to show that (3) implies (1) it is enough to use (\ref{relationindependence}) and the fact that $\theta(u,B) = \theta(u_1,B_{11})\theta(u_2,B_{22})$.

To conclude, suppose that $X$ is a discrete Gaussian with real parameters $u,B$. It is clear that (1) implies (4). We will show here that (4) implies (1). To do this, let $Y_1,Y_2$ be two independent discrete Gaussians with real parameters which have the same mean vector and covariance matrix as $X_1,X_2$ respectively. Then we see from (3) that $Y=(Y_1,Y_2)$ is a discrete Gaussian distribution with real parameters. Moreover, since $X_1,X_2$ are uncorrelated, $Y$ and $X$ have the same mean and covariance matrix. Then, Theorem \ref{mainthm} shows that $X,Y$ have the same distribution, and we are done.   
\end{proof}

 This result tells us that two joint discrete Gaussians are independent if and only if the matrix parameter $B$ is diagonal (modulo $\ZZ i$). This evokes the special property of continuous Gaussians that for the two components $X_1,X_2$ of a Gaussian vector $(X_1,X_2)$ one has that $X_1,X_2$ are independent if and only if their joint covariance $\Sigma_{12}$ is zero. 

\begin{remark}
The abelian variety $A_{B}$ corresponding to a parameter of the form $B = \left( \begin{smallmatrix} B_{11} & 0 \\ 0 & B_{22}  \end{smallmatrix} \right)$
is a product of abelian varieties of smaller dimension: $A_{B} = A_{B_{11}} \times A_{B_{22}}$. This is suggestive of the relation in algebraic statistics between statistical independence of finite random variables and the Segre embedding of $\mathbb{P}^{n_1} \times \mathbb{P}^{n_2}$.
\end{remark}

\section{Characteristic function and moments}

In this section we derive various probabilistic aspects of the discrete Gaussian, such as its characteristic function, its moments, its cumulants and its entropy. These already appear in the computer science literature: see for example \cite[Lemma 2.8]{MicPei} or \cite[Formulas (6),(7)]{OdedNoah}. These expressions have a simple description in terms of the Riemann theta function and we feel it is worth collecting them here in a systematic way.

We will use the following notation: if $a=(a_1,\dots,a_g)$ is a multi-index, with $a_i\in  \mathbb{N}, a_i\geq 0$, then we set $|a|=a_1+\dots+a_g$ and $a!=a_1!\dots a_g!$. Also, if $a,b\in \mathbb{N}^g$, we say that $a\leq b$ if $a_i\leq b_i$ for all $i=1,\dots,g$.

If $v=(v_1,\dots,v_g) \in \mathbb{C}^g$ is a vector of complex numbers we set $v^a = v_1^{a_1}\dots v_g^{a_g}$. If $f(u_1,\dots,u_g)$ is a holomorphic function, or a formal power series,  we set
\begin{equation}
 \qquad D^a_uf = \frac{\partial ^{a_1}f}{\partial u_1^{a_1}}\cdot\ldots\cdot\frac{\partial ^{a_g}f}{\partial u_g^{a_g}}, \quad D_u f = \left( \frac{\partial f}{\partial u_1}, \dots, \frac{\partial f}{\partial u_g} \right), \quad (D_uf)^a = \left(\frac{\partial f}{\partial u_1}\right)^{a_1}\cdot\ldots\cdot\left(\frac{\partial f}{\partial u_g}\right)^{a_g}
\end{equation}
and we denote by $D^a_u(\log f)$ the higher logarithmic derivatives of $f$. Recall that they are computed by taking formal derivatives of $\log f$: for example, if $f=f(u)$ is a function of one variable we have
\begin{equation}
D^1_u(\log f) = \frac{f'}{f}, \qquad D^2_u(\log f) = \frac{f\cdot f'' - f'^2}{f^2}.
\end{equation}
If $X = (X_1,\dots,X_g)$ is a random variable (possibly with a complex-valued distribution) we denote the corresponding means as $\mu_i = \mathbb{E}[X_i]$ and the mean vector as $\mu=(\mu_1,\dots,\mu_g)^t$. Moreover for each multi-index $a\in \mathbb{N}^g$ the higher moments are
\begin{equation}
\mu_a[X] = \mu_{a_1,\dots,a_g}[X] = \mathbb{E}[X_1^{a_1}\cdot\ldots\cdot X_g^{a_g}] \qquad a_i \in \mathbb{N}
\end{equation}
and the  higher central moments are
\begin{equation}
m_a[X] = \mu_{a}[X-\mu] = \mathbb{E}[(X_1-\mu_1)^{a_1}\cdot \ldots \cdot (X_g - \mu_g)^{a_g}].
\end{equation}
The moments are encoded by the characteristic function of $X$: this is the formal power series 
\begin{equation}
  \phi_{X}(v) = \mathbb{E}[ e^{i v^t X}] =  \sum_{a\in \mathbb{N}^g} \frac{i^{|a|}}{a!} \mu_a[X] v^a.
\end{equation}
We also consider the cumulants $\kappa_a[X]$, defined through a generating function, namely,
\begin{equation}
\sum_{a\in\mathbb{N}^g} \frac{\kappa_a[X]}{a!}v^a = \log \mathbb{E}[e^{v^t X}].
\end{equation}
We rephrase this by saying that the cumulants correspond to the logarithmic derivatives of the formal power series $v\mapsto \mathbb{E}[e^{v^t X}]$ evaluated at $0$:
\begin{equation}
\kappa_a[X] = D_v^a(\log \mathbb{E}[e^{v^t X}])_{|v=0}.
\end{equation}
In the case of the discrete Gaussian distribution, these quantities can be easily expressed using the Riemann theta function. We sometimes suppress the explicit dependence of $u, B$ in $
\theta(u,B)$ and its derivatives for readability.

\begin{proposition}\label{explicitmoments}
Fix $(u,B) \in \mathbb{C}^g \times \mathbb{H}_g \setminus \Theta_g$. Then the characteristic function of $X_{(u,B)}$ equals
\begin{equation}
\phi_{X_{(u,B)}}(v) = \mathbb{E}[e^{iv^t X_{(u,B)}}] = \frac{\theta\left( u + \frac{i}{2\pi }v ,B \right)}{\theta(u,B)} =  \frac{1}{\theta} \sum_{a \in \mathbb{N}^g} \frac{i^{|a|}}{(2\pi  )^{|a|}}\frac{1}{a!}(D^a_u\theta)\cdot  v^a.
\end{equation}
Consequently, the moments, central moments and cumulants are given by
\begin{align}
  \mu_{a} [X_{(u,B)}] &= \frac{1}{(2\pi  )^{|a|}} \frac{1}{\theta} D^a_u\theta, \\
  m_{a}[X_{(u,B)}] &=  \frac{1}{(2\pi )^{|a|}} \frac{1}{\theta} \sum_{0\leq b\leq a} \binom{a}{b} \frac{(-1)^{|b|}}{\theta^{|b|}} (D_u \theta)^b D^{a-b}_u\theta, \\
  \kappa_{a}[X_{(u,B)}] &= \frac{1}{(2\pi )^{|a|}} D^a_u(\log \theta).
\end{align}
where the expressions on the right hand side are evaluated at $(u,B)$.
\end{proposition}
\begin{proof}
  By definition,
  \begin{align}
    \phi_{X_{(u,B)}}(v) &= \mathbb{E}[e^{iv^t X_{(u,B)}}] = \mathbb{E}\left[\e\left( \frac{i}{2\pi }v^t X_{(u,B)} \right)\right]  = \frac{1}{\theta}\sum_{n\in \mathbb{Z}^g}\e\left(- \frac{1}{2}n^t B n + n^t u \right)\e\left( \frac{i}{2\pi }v^t n \right)\\
                           & = \frac{1}{\theta} \sum_{n\in \mathbb{Z}^g}\e\left( -\frac{1}{2}n^t B n + n^t \left( u + \frac{i}{2\pi }v \right) \right) = \frac{\theta\left( u + \frac{i}{2\pi }v ,B \right)}{\theta(u,B)} \\
    & = \frac{1}{\theta} \sum_{a \in \mathbb{N}^g} \frac{i^{|a|}}{(2\pi  )^{|a|}}\frac{1}{a!}(D^a_u\theta)(u,B)\cdot  v^a
  \end{align}
  where the last equality follows from the Taylor expansion of $\theta(u+\frac{i}{2\pi }v,B)$ around $u$. The formula for $\mu_a$ is obtained by substituting $\phi_X(-iv)$. To get the central moments, we observe that if $\mu=\mu[X_{(u,B)}]$ then $m_a[X_{(u,B)}] = \mu_a[X_{(u,B)}-\mu]$. Hence, we get them from the characteristic function of $X_{(u,B)}-\mu$:
  \begin{align}
   \phi_{X_{(u,B)}-\mu}(v) & = \phi_{\mu}(-v)\phi_{X_{(u,B)}}(v) = e^{-iv^t \mu}\left( \sum_{c\in \mathbb{N}^g} \frac{i^{|c|}}{c!}\mu_c[X_{(u,B)}] v^c \right) \\
   & = \left( \sum_{b\in \mathbb{N}^g}\frac{i^{|b|}}{b!} (-1)^{|b|}\mu^b v^b \right) \left( \sum_{c\in \mathbb{N}^g} \frac{i^{|c|}}{c!}\mu_c[X_{(u,B)}] v^c \right) \\
   & = \sum_{a\in \mathbb{N}^g}\left( \sum_{b+c=a} \frac{i^{|a|}}{b! c!} (-1)^{|b|} \mu^b \cdot \mu_c[X_{(u,B)}] \right) v^a \\
   & = \sum_{a \in \mathbb{N}^g} \frac{i^{|a|}}{a!} \left( \sum_{ 0 \leq b \leq  a} \binom{a}{b}(-1)^{|b|}\mu^b \cdot \mu_{a-b}[X_{(u,B)}] \right) v^a.
  \end{align}
  Substituting $\mu=\frac{1}{\theta} \frac{1}{2\pi } D_u\theta$ and $\mu_{a-b}[X_{(u,B)}] = \frac{1}{\theta}\frac{1}{(2\pi )^{|a-b|}} D^{a-b}_u\theta$ (we suppress the evaluation at $(u,B)$ for readability) we get
\begin{align}
 \phi_{X_{(u,B)}-\mu}(v) & = \sum_{a \in \mathbb{N}^g} \frac{i^{|a|}}{a!} \left( \frac{1}{\theta} \frac{1}{(2\pi )^{|a|}}\sum_{ 0 \leq b \leq  a} \binom{a}{b}\frac{(-1)^{|b|}}{\theta^{|b|}}(D_u\theta)^b \cdot D^{a-b}_u\theta \right) v^a
\end{align}
which is what we want.

To conclude, we compute the cumulants. By definition we have
\begin{align*}
     \log \mathbb{E}[e^{v^t X_{(u,B)}}] 
    &= \log  \frac{1}{\theta}\sum_{n\in \mathbb{N}^g} \e\left(\frac{1}{2\pi } v^t n\right)\e\left(- \frac{1}{2}n^t B n + u^t n \right)\\
    &= \log  \frac{1}{\theta}\sum_{n\in \mathbb{N}^g}\e\left( -\frac{1}{2}n^t B n + \left(u + \frac{1}{2\pi }v \right)^t n \right) = \log \theta\left( u + \frac{v}{2\pi }, B \right) - \log \theta(u,B)
  \end{align*}
And taking derivatives with respect to $v$ we get
  \begin{equation}
   \kappa_a[X_{(u,B)}] = \frac{1}{(2\pi )^{|a|}} D^a_u(\log \theta),
 \end{equation}
 where we suppress again the evaluation at $(u,B)$.
\end{proof}

\begin{remark}[Mean and covariance]\label{momentcovariance}
  In particular, we can write down explicitly the mean vector and the covariance matrix of $X_{(u,B)}$. Indeed, the first moments always coincide with the first cumulants, so that
  \begin{equation}
   \mu_i[X_{(u,B)}] = \frac{1}{2 \pi } \frac{1}{\theta}\frac{\partial \theta}{ \partial u_i} = \frac{1}{2\pi }  \frac{\partial \log \theta }{\partial u_i} = \kappa_i[X_{(u,B)}].
  \end{equation}
  For the covariances we also have that the second central moments coincide with the second cumulants, so that
  \begin{equation}\label{covariances}
    m_{ij}[X_{(u,B)}] = \frac{1}{(2\pi )^2} \frac{1}{\theta^2}\left( \theta \frac{\partial^2 \theta}{\partial u_i \partial u_j} - \frac{\partial \theta}{\partial u_i}\frac{\partial \theta}{ \partial u_j} \right) =  \frac{1}{(2\pi )^2} \frac{\partial^2 \log \theta}{\partial u_i \partial u_j} = \kappa_{ij}[X_{(u,B)}].
  \end{equation}
  We can compare these formulas with those obtained from the theory of exponential families in Corollary \ref{cor:bij}. In particular, from the expression for the covariances, we get that
  \begin{equation}\label{heatequation}
\frac{\partial \theta}{\partial B_{ii}} = -\frac{1}{4 \pi } \frac{\partial^2\theta}{\partial u_i^2 } \quad  \text{ and } \quad
\frac{1}{2}\frac{\partial \theta}{\partial B_{ij}} = -\frac{1}{4 \pi } \frac{\partial^2\theta}{\partial u_i \partial u_j} \quad \text{ if } i<j  
\end{equation}
which is precisely the heat equation for the theta function \cite[Proposition 8.5.5]{BL}.
\end{remark}

\begin{remark}\label{parityfunctions}
As a consequence of  Lemma \ref{parity}, we see that the moments $\mu_a[X_{(u,B)}]$, the central moments $m_a[X_{(u,B)}]$ and the cumulants $\kappa_{a}[X_{(u,B)}]$ are even or odd functions of $u$, depending on whether $|a|$ is even or odd. In particular, we get immediately that $\mu[X_{(0,B)}]=0$.   
\end{remark}

We can also compute the entropy of  discrete Gaussian random variables.

\begin{proposition} \label{prop:entropy}
  Fix $(u,B) \in \mathbb{C}^g \times \mathbb{H}_g \setminus \Theta_g$, and set $\mu=\mu[X_{(u,B)}]$ and $\Sigma = \operatorname{Cov}[X_{(u,B)}]$. Then the entropy of $X_{(u,B)}$ is given by
  \begin{align}\label{entropydiscretegaussian}
    H[X_{(u,B)}] & = \log \theta - \frac{1}{\theta}\left( \langle u,D_u\theta \rangle + \frac{1}{2} \langle B,D_{B}\theta+\operatorname{diag}(D_{B}\theta) \rangle\right) \\
    & = \log \theta - 2\pi  \langle u,\mu \rangle + \pi  \langle B, \Sigma + \mu \mu^t \rangle 
 \end{align}
 where $\langle \cdot,\cdot \rangle $ denotes the standard scalar product.
\end{proposition}
\begin{proof}
  We compute:
  \begin{align*}
    H[X_{(u,B)}] &= - \left( \sum_{n\in \mathbb{Z}^g} \log \left( \frac{\e\left( -\frac{1}{2}n^tB n + n^tu \right)}{\theta(u,B)} \right)  \frac{\e\left( -\frac{1}{2} n^tB n + n^tu \right)}{\theta(u,B)}  \right) = \\
                    & \sum_{n\in \mathbb{Z}^g} \left(\log \theta(u,B) + \pi n^tB n - 2\pi  n^tu  \right)  \frac{\e\left( -\frac{1}{2}n^tB n + n^tu \right)}{\theta(u,B)}   = \\
    & \log \theta(u,B) + \frac{1}{\theta}\left( \pi \sum_{n\in\mathbb{Z}^g} ( n^tB n) \e\left(- \frac{1}{2}n^tB n + n^tu \right) - 2\pi \sum_{n\in\mathbb{Z}^g} ( n^tu) \e\left(- \frac{1}{2}n^tB n + n^tu \right)  \right).
  \end{align*}
  Now, differentiating the theta function term by term, we see that
  \begin{equation}
  \frac{\partial \theta}{\partial u_i} =  2\pi  \sum_{n\in \mathbb{Z}^g} n_i \e\left(- \frac{1}{2}n^tB n + n^tu \right)  
\end{equation}
whereas for $i<j$, 
\begin{equation}
\frac{\partial \theta}{\partial B_{ii}} = -\pi  \sum_{n\in \mathbb{Z}^g} n_i^2\, \e\left(- \frac{1}{2}n^tB n + n^tu \right) \quad  \quad
 \frac{\partial \theta}{\partial B_{ij}} = -\pi  \sum_{n\in \mathbb{Z}^g} 2n_in_j\, \e\left(- \frac{1}{2}n^tB n + n^tu \right).  \end{equation}
With this, the first equality in (\ref{entropydiscretegaussian}) follows. For the second, we know from Proposition \ref{explicitmoments} that $\frac{1}{\theta}D_u\theta = 2\pi \mu$, so we just need to check that $\frac{1}{\theta}\frac{D_B\theta + \operatorname{diag}(D_B\theta)}{2} = -\pi(\Sigma+\mu \mu^t)$. To do this, it is enough to apply Proposition \ref{explicitmoments} together with the heat equation for the theta function (\ref{heatequation})
\end{proof}

\begin{remark}
We should point out that, since $\theta$ is a complex function, the logarithm $\log \theta$ is defined just up to a constant. However, when $u,B$ are real, we can take the standard determination of the logarithm on $\mathbb{C}$, which is the one that coincides with the usual logarithm on the positive real axis. In that case the above formula gives the entropy of the real discrete Gaussian. 
\end{remark}

\subsection{Numerical examples} \label{rmk:comp} 
A nice consequence of the bijection in Corollary \ref{cor:bij} and the expressions in Proposition \ref{explicitmoments} is that we can compute numerically the real parameters $u,B$ from given moments $\mu,\Sigma$ by solving the gradient system   
\begin{align} 
  \mu & = \frac{1}{2\pi } \frac{1}{\theta}D_u\theta \label{syst1}  \\
  \Sigma + \mu \mu^t & = -\frac{1}{2\pi } \frac{1}{\theta}(D_{B}\theta + \operatorname{diag}(D_{B}\theta)) \label{syst2}.
  \end{align}
It provides an advantage with respect to parametrizations like in \cite{KEMP} where the constant is not easy to compute. Indeed, geometers have long been interested in computing the Riemann Theta function and there are several numerical implementations \cite{BOB}. Thus, in principle, the discrete Gaussian distribution we propose can be computed effectively for applications. 

In fact, in \texttt{SAGE} using the package \texttt{abelfunctions} and with the help of Lynn Chua (who has maintained and implemented similar functions), we were able to compute with numerical precision the corresponding $u,B$ for given choices of $\mu,\Sigma$. A desirable step following this work would be to have a fully efficient implementation for all $g\geq 1$, for instance as an \texttt{R} package. 

For example, let $g=1$  and consider the sample with 10 data points $(1, 0, 1, -2, 1, 2, 3, -2, 1, -1 )$ that measure some discrete error. Then the sample mean is $\hat{\mu}= 0.4$ and the sample standard deviation is $\hat{\sigma}=1.6465$. Assuming a univariate discrete Gaussian model and solving numerically the system \eqref{syst1}, \eqref{syst2} we get $\hat{u} = 0.023 $ and $\hat{B} = 0.0587 $ as the maximum likelihood estimates. 

Now suppose we want to know the corresponding real parameters $u,B$ for a ``standard'' discrete Gaussian on $\ZZ$ with mean
$\mu=0$ and variance $\sigma^2=1$. By the parity property  of Remark \ref{parity}, we must have $u=0$ (see also Remark \ref{parityfunctions}). Solving numerically for $B$ in \eqref{syst2}, we get 
\begin{equation}\label{approximate2pi}
B \approx 0.1591549.
\end{equation} 
We can generalize this to a standard discrete Gaussian on $\mathbb{Z}^g$ with mean vector $\mu=(0,0,\ldots,0)$ and identity covariance matrix $\Sigma= Id$. Indeed, the parity property gives again $u=(0,0,\ldots,0)$. Moreover,  Proposition \ref{independence} implies that
\begin{equation}
B \approx \left( \begin{matrix}
0.1591549 & 0 &0 \\
0 & \ddots & 0 \\
0 & 0 &0.1591549 
\end{matrix} \right).
\end{equation} 

We can also go in the other direction and compute the mean vector and covariance matrix starting from the parameters $u$ and $B$. As an explicit example we take the Fermat quartic $C = \{ X_0^4+X_1^4+X_2^4 = 0 \}$ in $\mathbb{P}^2$. This is a smooth curve of genus $3$, so that its Jacobian is a principally polarized abelian variety of dimension $g=3$. We can compute the corresponding matrix parameter $B$ in \texttt{MAPLE} by using the function \texttt{periodmatrix}. We get that
\begin{equation}
B = \begin{pmatrix} 1-i & -0.5+0.5 i & 0.5 + 0.5 i \\ -0.5+0.5i & 1.25-0.25i & -0.75 + 0.25 i \\ 0.5 + 0.5 i & -0.75 + 0.25 i & 0.75 + 0.25i   \end{pmatrix}.
\end{equation}
Then, we can compute the mean vector and the covariance matrix of the discrete Gaussian on $\mathbb{Z}^3$ with parameters $u=0$ and $B$. Since $u=0$, the mean vector is also zero by Remark \ref{parityfunctions}. We can compute the covariance matrix numerically and we get
\begin{small}
\begin{equation*}
  \Sigma \approx
  \begin{pmatrix}
    -3.2885 + 8.1041 \cdot  10^{-9}i     & 0.1335 \cdot 10^{-6} + 0.8432 \cdot10^{-7}i & 5.989 + 0.635 \cdot 10^{-7}i \\
    0.1335\cdot 10^{-6} + 0.8432 \cdot 10^{-7}i & -6.283 + 5.696 i  & -6.283 + 5.696  i \\
    5.989 + 0.635\cdot 10^{-7}i & -6.283 + 5.696 i & -18.262 + 5.696 i 
  \end{pmatrix}.
\end{equation*}
\end{small}

\begin{remark}
From the approximation \eqref{approximate2pi}, one could think that the exact parameter for the ``standard" discrete Gaussian on $\mathbb{Z}$ should be  $B=\frac{1}{2\pi}$. Unfortunately, this is not the case: indeed, we know from \eqref{covariances} that $\operatorname{Var}[X_{(0,B)}] = \frac{1}{(2\pi)^2}\frac{\partial^2\log \theta}{\partial u^2}(0,B)$. Now, the classical Jacobi identity \cite[Table V, p. 36]{MUM1} for the theta function gives
\begin{equation}
\theta\left( \frac{u}{iB} , \frac{1}{B} \right) = \sqrt{B}e^{-\frac{\pi}{B} u^2} \theta(u,B). 
\end{equation}
Hence, taking the second logarithmic derivative on both sides, evaluating at $u=0$ and dividing by  $\frac{1}{(2\pi)^2}$ we get
\begin{equation}
\frac{1}{(2\pi)^2}\frac{\partial^2 \log \theta}{\partial u^2}\left(0,
\frac{1}{B} \right) = \frac{B}{2\pi} - \frac{1}{(2\pi)^2}\frac{\partial^2 \log \theta}{\partial u^2}(0,B).
\end{equation}
Setting $B=2\pi$ we get:
\begin{equation}
\operatorname{Var}[X_{(0,\frac{1}{2\pi})}] = 1 - \operatorname{Var}[X_{(0,2\pi)}]
\end{equation}
and since $X_{(0,2\pi)}$ is a random variable on $\mathbb{Z}$ which is not constant, it follows that $\operatorname{Var}[X_{(0,2\pi)}]\ne 0$ so that $\operatorname{Var}[X_{(0,\frac{1}{2\pi})}] \ne 1$. It would be interesting to learn more about the true constant $B$.
\end{remark}

\section{Statistical theta functions}

Now, let's fix a parameter $B\in \mathbb{H}_g$ and let $(A_{B},\Theta_{B})$ be the corresponding principally polarized abelian variety. We have seen before that discrete Gaussians $X_{(u,B)}$ correspond, up to translations, to points in the open abelian variety $A_{B}\setminus \Theta_{B}$. Hence, every statistical function that is invariant under translations of random variables gives a well-defined function on $A_{B}\setminus \Theta_{B}$. Natural choices of such functions are the higher central moments and the higher cumulants. We show here that they actually define meromorphic functions on $A_{B}$.

\begin{proposition}\label{statisticalthetafunctions}
  Fix $B\in\mathbb{H}_g$ and $a\in \mathbb{N}^g$ with $|a|>1$. Then the higher central moment $m_a$ and the cumulant $\kappa_a$ define  meromorphic functions on $A_{B}$ with poles only along the theta divisor $\Theta_{B}$ of order at most $|a|$. Hence they can be seen as global sections in $H^0(A_B, |a|\Theta_{B})$.
 Moreover, they have poles precisely of order $|a|$ if and only if $(D_u\theta)^a$ is not identically zero along $\Theta_B$. In particular, this happens if $\Theta_{B}$ is irreducible.
\end{proposition}
\begin{proof}

First,  we consider the higher central moments. The explicit expression for $m_a$ given in Proposition \ref{explicitmoments}, shows that $m_a$ is a meromorphic function on $A_{B}$ with poles at most along the theta divisor $\Theta_{B}$. Moreover, looking again at the explicit form, we see that the summand corresponding to $b$ has a pole of order at most $|b|+1$. The maximum value possible for $|b|$ is $|a|$ and it corresponds to the unique case $b=a$. The corresponding summand is
\begin{equation}
\frac{1}{(2\pi )^{|a|}} \frac{(-1)^{|a|}}{\theta^{|a|+1}}(D_u\theta)^a D_u^0 \theta = \frac{1}{(2\pi )^{|a|}} \frac{(-1)^{|a|}}{\theta^{|a|+1}}(D_u\theta)^a \theta = \frac{(-1)^{|a|}}{(2\pi i)^{|a|}} \frac{(D_u\theta)^a}{\theta^{|a|}},
\end{equation}
which has poles only along $\Theta_{B}$ of order at most $|a|$. This shows immediately that  $m_a$ has poles only along $\Theta_{B}$ of order at most $|a|$. For the cumulants $\kappa_{a}$ one can follow a similar reasoning by writing an explicit expression for $D^a_u \log \theta$, or one can use the fact that the cumulants are determined by the central moments through some explicit linear expressions \cite[eqn (7)]{MOMVAR}. 

Now we ask ourselves when do $m_a$ and $\kappa_a$ attain poles of order $|a|$. Looking at the expression of Proposition \ref{explicitmoments}, we consider the summands with $|b|=|a|-1$: these are of the form $a-E_i$, where $i$ is an index such that $a_i>0$. Taking the sum over these we get
\begin{align}
& -\frac{1}{(2\pi )^{|a|}} \frac{(-1)^{|a|}}{\theta^{|a|}} \sum_{\{i | a_i>0\}} a_i (D_u\theta)^{a-E_i} D_u^{E_i} \theta  = -\frac{1}{(2\pi )^{|a|}} \frac{(-1)^{|a|}}{\theta^{|a|}} \sum_{i} a_i (D_u\theta)^{a}  \\
= & -\frac{(-1)^{|a|}}{(2\pi )^{|a|}} \frac{(D_u\theta)^{a}}{\theta^{|a|}} \sum_{i} a_i = \frac{(-1)^{|a|}}{(2\pi )^{|a|}} \frac{(D_u\theta)^{a}}{\theta^{|a|}}(-|a|).
\end{align}
So we can write
\begin{equation}\label{highestpole}
\mu_a =  \frac{(-1)^{|a|}}{(2\pi )^{|a|}} \frac{(D_u\theta)^{a}}{\theta^{|a|}}(1-|a|) + f_a
\end{equation}
where $f_a$ is a meromorphic function with poles only along $\Theta_{B}$ of order strictly less than $|a|$. Since we are assuming $|a|>1$, we see that $1-|a|\ne 0$, hence $m_a$ has a pole of order $|a|$ along $\Theta_{B}$ if and only if $(D_u\theta)^{a}$ is not identically zero along $\Theta_{B}$. For the cumulants, one can follow a similar reasoning by using an explicit expression for $D^a_u \log \theta$ or the formula relating cumulants with central moments \cite{WILL}.  

To conclude, we need to show that if the theta divisor $\Theta_{B}$ is irreducible, then $D^a_u\theta$ does not vanish identically along it. Since $\Theta$ is irreducible and  $D^a_u\theta$ is a product of terms of the form $\frac{\partial \theta}{\partial u_i}$ it is enough to check that each one of these terms does not vanish identically along $\Theta_{B}$. But this follows from \cite[Proposition 4.4.1]{BL}.
\end{proof}

\begin{remark}
We see that the cumulants correspond to the logarithmic derivatives of the theta function, which are classical meromorphic functions on abelian varieties \cite[Section I.6, Method III]{MUM1}. However, the central moments give, to the best of our knowledge, new meromorphic functions. 
\end{remark}

\subsection{Statistical maps  of abelian varieties}

By definition, abelian varieties can be embedded into projective space. In algebraic geometry it is of great interest to produce rational maps from an abelian variety $A_B$ to different projective spaces. We show  that this is possible using statistical data, namely central moments and cumulants.

Let us fix $B\in \mathbb{H}_g$ and consider all the higher central moments $ [m_{a}]_{|a|\leq d,|a|\ne 1 }$ of order $d\geq 2$. Then, according to Proposition \ref{statisticalthetafunctions}, we get rational maps  $A_{B} \dashrightarrow \mathbb{P}^{N_{g,d}}$ where $ N_{g,d} = \binom{g+d}{d}-g-1$. Observe that we do not consider the linear central moments, because they are always zero by definition.
These maps can be seen as similar to the moment varieties for Gaussians and their mixtures studied in \cite{MOMVAR},\cite{IDENT}. 
We proceed to study here the geometry of these maps and derive some statistical consequences.

First, as in \cite{MOMVAR}, we can replace the central moments $m_a[X_{(u,B)}]$ with the cumulants $\kappa_a[X_{(u,B)}]$, which are often easier to work with. Moreover, we want to work with holomorphic functions instead of meromorphic functions, so that we multiply the coordinates everywhere by $\theta^d$.  The maps that we want to consider are the following:

\begin{definition}[Statistical maps of abelian varieties]
With the above notation, we define the \textit{statistical maps}   
\begin{equation}
 \phi_{d,B}\colon A_{B} \dashrightarrow \mathbb{P}^{N_{g,d}}, \qquad \phi_{d,B} = [\theta^d\kappa_{a}]_{|a|\leq d,|a|\ne 1 }, \qquad N_{g,d} = \binom{g+d}{d}-g-1. 
\end{equation}
\end{definition}

To study the geometry of these maps it will be useful to consider their restriction to the theta divisor. It turns out that this can be described in terms of the classical Gauss map of the theta divisor. We recall briefly the definition of the Gauss map here, and for more details one can look at \cite[Section 4.4]{BL}.

First, if we differentiate the quasiperiodicity relation (\ref{quasiperiodicitytheta}) we get:
\begin{equation}
\frac{\partial \theta}{\partial u_i}(u+im+B n, B) = 2\pi  n_i\, \e\left(  \frac{1}{2}n^tB n + n^tu \right)\theta(u,B) + \e\left( \frac{1}{2}n^tB n + n^tu \right) \frac{\partial \theta}{\partial u_i}(u, B).
\end{equation}
In particular, when  $\theta(u,B)=0$ this gives
\begin{equation}
\frac{\partial \theta}{\partial u_i}(u+im+B n, B) = \e\left( \frac{1}{2}n^tB n + n^tu \right) \frac{\partial \theta}{\partial u_i}(u, B).
\end{equation}
Hence, on the theta divisor $\Theta_{B} \subseteq A_{B}$  we have a well defined rational map
\begin{equation}\label{gaussmap}
\gamma \colon \Theta_{B} \dashrightarrow \mathbb{P}^{g-1} \qquad \gamma = \left[ \frac{\partial \theta}{\partial u_1}, \dots, \frac{\partial \theta}{\partial u_g} \right]
\end{equation}
called the \textit{Gauss map}. By construction, the Gauss map is not defined at the points of $\Theta_{B}$ where all the partial derivatives vanish. This means that the Gauss map is not defined precisely  at the singular points of $\Theta_{B}$. One can show that this is the same map induced by the complete linear system $H^0(\Theta_{B},\mathcal{O}_{\Theta_{B}}(\Theta_{B}))$. 

In the next lemma we show how to describe the restriction of $\phi_{d,B}$ to $\Theta_{B}$ in terms of the Gauss map.

\begin{lemma}\label{restrictiontotheta}
The restriction ${\phi_{d,B}}_{|\Theta_{B}}\colon \Theta_{B} \dashrightarrow \mathbb{P}^{N_{g,d}}$ corresponds to the composition
\begin{equation}
\Theta_{B} \overset{\gamma}{\dashrightarrow} \mathbb{P}^{g-1} \overset{v_d}{\hookrightarrow} \mathbb{P}^{\binom{g-1+d}{d}-1} \hookrightarrow \mathbb{P}^{N_{g,d}}
\end{equation}
where $v_d$ is the $d$-th Veronese embeddding of $\mathbb{P}^{g-1}$, and the last map is the  linear embedding in the last $\binom{g-1+d}{d}$ coordinates that sets all the others to zero.
\end{lemma}
\begin{proof}
  By definition, the coordinates of $\phi_{d,B}$ are given by $\theta^d \kappa_{a}$, for $|a|\leq d, |a|\ne 1$. Reasoning as in the proof of Proposition \ref{statisticalthetafunctions},  we can write them in the form
\begin{equation}
\theta^d\kappa_a =  C \cdot  \frac{\theta^d}{\theta^{|a|}}(D_u\theta)^{a} + \theta^d g_a
\end{equation}
where $C$ is a nonzero constant and the $g_a$ are meromorphic functions with poles only along $\Theta_{B}$ and of order strictly smaller than $|a|$. In particular, since $|a|\leq d$, it follows that the functions $\theta^d g_a$ vanish identically along the theta divisor. The same is true for the functions $\frac{\theta^d}{\theta^{|a|}}(D_u\theta)^{a}$, when $|a|<d$. Hence, when we restrict the $\theta^d\kappa_a$ to the theta divisor $\Theta_{B}$, the only nonzero terms are those corresponding to $|a|=d$. In this case, we get
\begin{equation}\label{highestpolecumulant}
(\theta^d\kappa_a)_{|\Theta_{B}} =  C\cdot (D_u\theta)^{a}_{|\Theta_{B}},
\end{equation}
The common nonzero constant $C$ does not matter since we are taking a map into projective space.
Hence, we just need to consider the functions $\{(D_u\theta)^a_{|\Theta_{B}}\}_{|a|=d}$. However,  since the $(D_u\theta)^a$ correspond to taking monomials of degree $d$ in the partial derivatives $D_{u_i}\theta$, this is the same as the $d$-th Veronese embedding composed with the Gauss map, and we are done.
\end{proof}

Using this lemma we can draw consequences on the maps $\phi_{d,B}$. For example, it is immediate to see where the $\phi_{d,B}$ are defined.

\begin{corollary}\label{globalgeneration}
  The rational map $\phi_{d,B}\colon A_{B} \dashrightarrow \mathbb{P}^{N_{g,d}}$ is not defined precisely at the singular points of $\Theta_{B}$.
  In particular, it is defined everywhere if and only if $\Theta_{B}$ is smooth.
\end{corollary}
\begin{proof}
  Since the first coordinate of the map is $\theta^d$, the map is defined everywhere outside of $\Theta_{B}$. Moreover, Lemma \ref{restrictiontotheta} shows that the restriction of the map to $\Theta_{B}$ is essentially given by the Gauss map composed with an embedding. Hence, it is not defined if and only if the Gauss map is not defined, which happens precisely at the singular points of $\Theta_{B}$. 
\end{proof}

We usually want maps that are nondegenerate, in the sense that its image cannot be contained in a hyperplane. We now give a sufficient condition for this not to happen to the maps $\phi_{d,B}$.

\begin{lemma}\label{nondegeneracy}
Suppose that the theta divisor $\Theta_{B}$ is irreducible. Then the image of the rational map $\phi_{d,B}\colon A_{B} \dashrightarrow \mathbb{P}^{N_{g,d}}$ is not contained in a hyperplane.
\end{lemma}
\begin{proof}
  To show that the image is not contained in any hyperplane is the same as proving that the functions $\theta^d\kappa_{a}$ for $|a|\leq d,|a|\ne 1$ are linearly independent. We are going to prove linear independence by induction on $d$. The base cases $d=0$ and $d=1$ are just $\theta^d$. Now, assume we have a linear relation
  \begin{equation}
   \lambda_0\theta^d + \sum_{|a|=2}\lambda_a \theta^d\kappa_{a} + \dots + \sum_{|a|=d}\lambda_a\theta^d\kappa_{a} = 0, \qquad \text{ for certain } \lambda_a\in\mathbb{C}.
 \end{equation}
 By induction hypothesis, it is enough to prove that $\lambda_a=0$ for all $|a|=d$.  Since $\Theta_{B}$ is irreducible, Lemma \ref{statisticalthetafunctions} shows that each $\kappa_a$ has poles of order exactly $|a|$ along $\Theta_{B}$. Hence, if we restrict the above relation to $\Theta_{B}$, we get that $(\theta^d\kappa_a)_{|\Theta_{B}}=0$ for all $|a|<d$, and we are left with  
\begin{equation}
\sum_{|a|=d}\lambda_a(\theta^d\kappa_{a})_{|\Theta_B} = 0.
\end{equation}
It is then enough to show that the functions $(\theta^d\kappa_{a})_{|\Theta_B}$, for $|a|=d$, are linearly independent.

Equivalently,  we need to show that the image of the map $\Theta_{B}\dashrightarrow \mathbb{P}^{\binom{g-1+d}{d}-1}$ that these functions induce is not contained in any hyperplane. We know from Lemma \ref{restrictiontotheta} that this map is a composition
\begin{equation}
\Theta_{B} \overset{\gamma}{\dashrightarrow} \mathbb{P}^{g-1} \overset{v_d}{\hookrightarrow} \mathbb{P}^{\binom{g-1+d}{d}-1} 
\end{equation}
where $\gamma$ is the Gauss map and $v_d$ is a Veronese embedding. Since  $\Theta_{B}$ is irreducible, the Gauss map is dominant \cite[Proposition 4.4.2]{BL} and the image of the Veronese map is also not contained in any hyperplane. Hence, the image of the composite map is not contained in any hyperplane, and we are done.
\end{proof}

\begin{remark}\label{projection}
We can give a geometric interpretation of this result, which will help us study the maps $\phi_{d,B}$.  By construction,  $\phi_{d,B}$ is defined by a subset of functions in $H^0(A_{B},d\Theta_{B})$. It is known \cite[Proposition 4.1.5]{BL} that if $d\geq 2$, the complete linear system $H^0(A_{B},d\Theta_{B})$ induces everywhere defined maps
\begin{equation}
\psi_{d,B}\colon A_{B} \longrightarrow \mathbb{P}^{d^g-1}.
\end{equation}

If $\Theta_{B}$ is irreducible, Lemma \ref{nondegeneracy} says that our maps $\phi_{d,B}$ can be realized as a composition
\begin{equation}\label{linearprojection}
A_{B} \overset{\psi_{d,B}}{\longrightarrow} \mathbb{P}^{d^g-1} \overset{\pi}{\dashrightarrow} \mathbb{P}^{N_{g,d}}
\end{equation}
where the second map is a projection from a linear space. Hence, we can study $\phi_{d,B}$ through the maps  $\psi_{d,B}$, which are well-known for algebraic geometers. For example, the fibers of the map $\psi_{2,B}$ correspond exactly to the points $u,-u$ exchanged by the involution. In contrast, the maps $\psi_{d,B}$ are all closed embeddings for $d\geq 3$.
\end{remark}

We close this subsection with the following proposition.

\begin{proposition}\label{statisticgeometry}
Let $F\colon A_{B} \to A_{B'}$ be an isomorphism of polarized abelian varieties of dimension $g$. Then for every $d$ there exists a linear isomorphism $Q\colon \mathbb{P}^{N_{g,d}} \to \mathbb{P}^{N_{g,d}}$ such that 
\begin{equation}
\phi_{d,B'}\circ F = Q \circ \phi_{d,B}
\end{equation}
\end{proposition}
\begin{proof}
If two polarized abelian varieties $(A_B,\Theta_{B})$ and $(A_{B'},\Theta_{B'})$ are isomorphic, we know from the discussion in Section \ref{familiesabvar} that it must be that $B' = MB$ for a certain $M\in Sp(2g,\mathbb{Z})$, and then the isomorphism is given by $F_{M,B}$ as in (\ref{isomorphismMtau}), eventually composed with a translation. Then, we can compute the relation between the logarithmic derivatives of $\theta(u,B)$ and $\theta(F_{M,B}(u),MB)$, using the full Theta Transformation Formula \cite[Theorem 8.6.1]{BL} , and the conclusion follows. 
\end{proof}

What this result is saying is that when we take two isomorphic polarized abelian varieties $(A_{B},\Theta_{B}) \cong (A_{B'},\Theta_{B'})$, then the images of the two maps $\phi_{d,B},\phi_{d,B'}$ differ only by a linear change of coordinates.  Hence, we can say that the image of $\phi_{d,B}$ depends only on the isomorphism class of the polarized abelian variety $(A_B,\Theta_B)$.

In the next two subsections we study in more detail the univariate case $g=1$ and the bivariate case $g=2$.

\subsection{The case $g=1$}
We first consider univariate discrete Gaussians, which correspond to elliptic curves on the abelian varieties side. In this case, we are going to see that actually, the maps $\phi_{d,B}$ coincide with the well-studied maps $\psi_{d,B}$. This means that \textit{elliptic normal curves are parametrized by cumulants of univariate discrete Gaussians}.

More precisely, fix a $B \in \mathbb{H}_1$, and for every $d\geq 2$, consider the map
\begin{equation}
\phi_{d,B}\colon A_{B} \longrightarrow \mathbb{P}^{d-1}, \qquad \phi_{d,B} = [\theta^d, \theta^d\kappa_2, \dots, \theta^d\kappa_d].
\end{equation}
For elliptic curves, the theta divisor consists of a single point, so that it is immediately smooth and irreducible. By Corollary \ref{globalgeneration}, the statistical map $\phi_{d,B}$ is defined everywhere, and according to Remark \ref{projection}, the map is a composition
\begin{equation}
A_{B} \overset{\psi_{d,B}}{\longrightarrow} \mathbb{P}^{d-1} \overset{\pi}{\dashrightarrow} \mathbb{P}^{N_{1,d}}
\end{equation}
where the second map is a linear projection. But $N_{1,d} = d-1$, so that the projection $\pi$ must be the identity. Hence, we can identify $\phi_{d,B}$ with $\psi_{d,B}$. This means the following:
\begin{itemize}
\item $d=2$:  the map $\phi_{2,B}\colon A_{B} \to \mathbb{P}^1$ is surjective of degree $2$. The fibers correspond to opposite points $u, -u$ of $A_{B}$. In particular, it is ramified exactly at the four points of order two in $A_{B}$.
\item $d \geq 3$: the map  $\phi_{d,B} \colon A_{B} \hookrightarrow \mathbb{P}^{d-1}$ is a closed embedding. The image is usually called an elliptic normal curve of degree $d$. In particular, the map $\phi_{3,B}$ embeds $A_{B}$ as a smooth cubic curve in $\mathbb{P}^2$.
\end{itemize}

This sheds light on the relation between univariate discrete Gaussians and their moments. Theorem \ref{mainthm} tells us that when the parameters $u,B$ are real, the discrete Gaussians are uniquely determined by the moments up to order two. In contrast, we are going to show now that arbitrary (that is, complex) univariate discrete Gaussians of a fixed parameter $B$ are determined by the moments up to order three.  

\begin{theorem}\label{thirdmomentellcurves}
  Fix $B \in \mathbb{H}_1$ and let $X_{(u,B)},X_{(u',B)}$ be two discrete Gaussians with the same parameter $B$. If the first three moments coincide, then they have the same distribution. Furthermore, if we consider discrete Gaussians with different parameters $B$, the first two moments are not enough in general.
\end{theorem}
\begin{proof}
We prove now the first part. If the first three central moments coincide, then the cumulants $\kappa_2,\kappa_3$ coincide as well. This means that the two points $u,u'\in A_{B}$ have the same image under the map $\phi_{3,B}\colon A_{B} \to \mathbb{P}^2$. But we know from the previous discussion that this map is injective, hence, it must be that $u,u'$ are the same point in $A_{B}$, which is to say that $u'=u+im+B n$ for certain $n,m \in \mathbb{Z}$. Then, Proposition \ref{quasiperiodicitydg} shows that $X_{(u',B)} \sim X_{(u,B)} + n$, but since $X_{(u,B)},X_{(u',B)}$ have the same mean, it must be that $n=0$, and we conclude. 
\end{proof}

A statistical consequence is that all higher moments of a Gaussian distribution are uniquely determined by the first three. We can make this explicit using the equation of the cubic curve $\phi_{3,B}(A_{B})\subseteq \mathbb{P}^2$.

To compute this equation we set some notations. For every fixed $B \in \mathbb{H}_1$, we consider the three points $0,\frac{1}{2}i,\frac{B}{2} \in \mathbb{C}$. If we take their images in $A_{B}$, these give three of the $2$-torsion points of $A_{B}$. The other torsion point is given  by the image of $\frac{1}{2}i+\frac{B}{2}$, which coincides with the theta divisor $\Theta_{B}\subseteq A_{B}$ \cite[Lemma 4.1]{MUM1}. Now, we define 
\begin{equation}
e_1(B) := D^2_{u}(\log \theta(u,B))\left(0,B \right), \quad e_2(B) := D^2_{u}(\log \theta(u,B))\left( \frac{1}{2}i, B \right)
\end{equation}
$$e_3(B) := D^2_{u}(\log \theta(u,B))\left( \frac{B}{2}, B \right) $$
and the quantities
\begin{equation}
a(B) = \frac{1}{\pi^2}(e_1+e_2+e_3), \qquad
  b(B) = -\frac{1}{4\pi^4}(e_1e_2+e_1e_3 + e_2e_3), \qquad
  c(B) = \frac{1}{16\pi^6}e_1e_2e_3.
\end{equation}

\begin{proposition}\label{equationcubic}
  Fix $B\in\mathbb{H}_1$. Then the image of the map $\phi_{3,B}\colon A_{B} \to \mathbb{P}^2_{[X_0,X_1,X_2]}$ is the cubic
  \begin{equation}
   X_0X_2^2 = -4X_1^3 + a(B)X_0X_1^2 + b(B)X_0^2X_1 + c(B)X_0^3.
 \end{equation}
 Equivalently, for each $(u,B) \in \mathbb{C}\times \mathbb{H}_1\setminus \Theta_1$ we have the relation
  \begin{equation}
   \kappa_3^2 = -4\kappa_2^3 + a(B)\kappa_2^2 + b(B)\kappa_2 + c(B).
  \end{equation}
\end{proposition}
\begin{proof}
  For notational clarity we set $\nu_a = (2\pi )^a\kappa_a = D_u^a\log\theta$. Now, we fix $B$ and we consider the meromorphic function on $A_{B}$ given by $\nu_3^2+4\nu_2^3$: we can check, for example by a computer, that this has a pole of order at most four along $\Theta_{B}$. Let $H^0(A_{B},4\Theta_{B})$ be the space of meromorphic functions on $A_{B}$ with poles at most of order $4$ along $\Theta_{B}$. This space has dimension $4$ by the classical theorem of Riemann-Roch, and the functions $1,\nu_2,\nu_3,\nu_2^2$ are linearly independent elements in $H^0(A_{B},4\Theta_{B})$, since they have poles of different order along $\Theta_{B}$. Hence, they form a basis of the space, so that
  \begin{equation}
  \nu_3^2 = -4\nu_2^3+a(B)\cdot\nu_2^2+b(B)\cdot\nu_2+c(B)+d(B)\cdot\nu_3
\end{equation}
for certain $a(B),b(B),c(B),d(B)$ independent of $u$. Now we observe that the functions $\nu_3^2$ and $-4\nu_2^3+a(B)\nu_2^2+b(B)\nu_2+c(B)$ are even, whereas the function $d(B)\cdot\nu_3$ is odd. Hence, it must be that $d(B)=0$. The equation reduces to:
  \begin{equation}\label{cubic}
  \nu_3^2 = -4\nu_2^3+a(B)\nu_2^2+b(B)\nu_2+c(B).
\end{equation}
Now, consider the map $\phi_{2,B} = [1,\nu_2]\colon A_{B} \to \mathbb{P}^1$: we know from a previous discussion that this is ramified precisely at the two torsion points of $A_{B}$.  Since $D_u(\nu_2)=\nu_3$, this means precisely that $\nu_3(0)=\nu_3(\frac{1}{2}i)=\nu_3(\frac{B}{2})=0$. Plugging this into (\ref{cubic}), we get that $e_1=\nu_2(0),e_2=\nu_2(\frac{1}{2}i),e_3=\nu_2(\frac{B}{2})$ are roots of the polynomial $-4x^3+a(B)x^2+b(B)x+c(B)$. Then, we can write
\begin{equation}\label{abc}
a(B) = 4(e_1+e_2+e_3), \qquad b(B)=-4(e_1e_2+e_1e_3+e_2e_3), \qquad c(B)=4e_1e_2e_3.
\end{equation}
To conclude, it is enough to replace $\nu_a=(2\pi )^{a}\kappa_a$.
\end{proof}

\begin{remark}
  This relation tells us that the second cumulant determines the third cumulant up to a sign. We can also get the higher cumulants as follows: using the same notation as in the previous proof, consider the equation
  \begin{equation}
  \nu_3^2 = -4\nu_2^3+a(B)\nu_2^2+b(B)\nu_2+c(B).
\end{equation}
Then, taking derivatives with respect to $u$ we get
\begin{equation}\label{quadric}
 \nu_4 = -6\nu_2^2 + a(B)\nu_2 + \frac{b(B)}{2}.
\end{equation}
Passing to the cumulants, this gives a formula to compute the fourth cumulant from the second. Differentiating again, we get a formula for the fifth cumulant and so on. All together, this gives explicit formulas to compute the higher moments, starting from the first three.

It is worth mentioning that the equation we obtain in Proposition \ref{equationcubic} is the same as the one obtained through the Weierstrass $\wp$-function \cite[Section 1.6]{MUM1}. In addition, the relations we obtain by differentiating can be interpreted through the Korteweg-de Vries equation, as in \cite[Section 1.6]{MUM1}.
\end{remark}

As an application of the cubic equation of Proposition \ref{equationcubic}, we are going to show that univariate discrete Gaussians of arbitrary parameters $u,B$ are not uniquely determined by their moments up to order two.

To do this, we consider the map
\begin{equation}
\Phi\colon (\mathbb{C} \times \mathbb{H}_1)\setminus \Theta_1 \to \mathbb{C}^2 \qquad (u,B) \mapsto (\mu, \Sigma)= \left( \mu[X_{(u,B)}], \operatorname{Var}[X_{(u,B)}] \right).
\end{equation}
This assigns to each couple of parameters $(u,B)$ the mean and the covariance of the corresponding discrete Gaussian distribution. By construction, this map factors through the space $\mathcal{G}_1$ of all univariate discrete Gaussian distributions. Hence, we get a map
\begin{equation}
\widetilde{\Phi}\colon \mathcal{G}_1 \to \mathbb{C}^2
\end{equation}
and we are asking whether this map is injective. To show that this is not the case, we conclude the proof of Theorem \ref{thirdmomentellcurves}.

\begin{proof}[Proof of Theorem \ref{thirdmomentellcurves}, second part]
We know from Remark \ref{remarkequivalence} that $\mathcal{G}_1$ is a complex manifold of dimension two. Hence,  if $\widetilde{\Phi}$ is injective, its differential should be everywhere an isomorphism between tangent spaces.  To check this, we may as well work with $\Phi$, for which we can compute the differential explicitly. We have
\begin{equation}
\Phi(u,B) = \left(\mu[X_{(u,B)}],\operatorname{Var}[X_{(u,B)}]\right) =  \left(\frac{1}{(2\pi)} \nu_1, \frac{1}{(2\pi )^2} \nu_2 \right)
\end{equation}
where $\nu_1,\nu_2$ are the same as in the proof of Proposition \ref{equationcubic}. Up to multiplication by nonzero constants, the differential of $\Phi$ is given by
\begin{equation}
D\Phi =  \begin{pmatrix} D^1_u\nu_1 & D_{B}^1\nu_1 \\ D^1_u\nu_2  & D^1_B \nu_2    \end{pmatrix} = \begin{pmatrix} \nu_2 & D_{B}^1\nu_1 \\ \nu_3  & D^1_B \nu_2    \end{pmatrix}.
\end{equation}
However, we see that $D_{B}^1\nu_1 = D_{B}^1D^1_u(\log \theta) = D^1_u D^1_B(\log \theta)$ and the same holds for $\nu_2$. Then, 
we can use the heat equation (\ref{heatequation}), and we get that
\begin{equation}
D^1_{B}(\log \theta) = \frac{D^1_{B}\theta}{\theta} = -\frac{1}{4\pi } \frac{D_u^2\theta}{\theta} = -\frac{1}{4\pi }\left( D^2_u(\log\theta) + D^1_u(\log\theta)^2 \right) = -\frac{1}{4\pi }(\nu_2+\nu_1^2)
\end{equation}
Thus, up to multiplying by a nonzero constant, we can write
\begin{align}
  \det D\Phi & = \det \begin{pmatrix} \nu_2 & \nu_3 + 2\nu_1\nu_2 \\ \nu_3  & \nu_4 + 2\nu_2^2 + 2\nu_1\nu_3    \end{pmatrix} \\
             & = \det \begin{pmatrix} \nu_2 & \nu_3  \\ \nu_3  & \nu_4 + 2\nu_2^2     \end{pmatrix} \\
  &  = \det \begin{pmatrix} \nu_2 & \nu_3 \\ \nu_3  & \nu_4 + 2\nu_2^2 \end{pmatrix} = \nu_2\nu_4 + 2\nu_2^3 -\nu_3^2. 
\end{align} 
Using the equations (\ref{cubic}) and (\ref{quadric}), we get that:
\begin{equation}
\nu_2\nu_4 + 2\nu_2^3 -\nu_3^2 = -\frac{b(B)}{2}\nu_2 + c(B)
\end{equation}
where $b(B),c(B)$ are as in the proof of Proposition \ref{equationcubic}. We want to show that this vanishes on some points of $\mathbb{C}\times \mathbb{H}_1\setminus \Theta_1$. It is enough to show that $b(B)$ is not identically zero. Indeed, if we fix a $B$ such that $b(B)\ne 0$, the function $-\frac{b(B)}{2}\nu_2 + c(B)$ is a meromorphic function on $A_{B}$ with poles of order two along the theta divisor $\Theta_{B}$, so that it must have zeroes on $A_{B}\setminus \Theta_{B}$. A \texttt{SAGE} computation reveals that $b(B)$ is indeed not identically zero (by evaluating its explicit expression \eqref{abc}), and we are done. 
\end{proof}

\subsection{The case $g=2$} In the case of dimension two the situation becomes  more complicated.

First, let's fix a parameter $B \in \mathbb{H}_2$ such that the theta divisor $\Theta_{B}$ is irreducible. In this case \cite[Corollary 11.8.2]{BL} $\Theta_{B}$ is also smooth: more precisely, $\Theta_{B}$ is a smooth curve of genus $2$ and $A_{B}$ is its Jacobian variety.

\begin{remark}\label{gaussmapabsurf}
It will be useful to single out some results about the Gauss map
\begin{equation}
\gamma\colon \Theta_{B} \longrightarrow \mathbb{P}^1
\end{equation}
in this situation. We have seen before that this map corresponds to the one induced by the line bundle $\mathcal{O}_{\Theta_{B}}(\Theta_{B})$. Since $\Theta_{B}$ is a smooth curve, the adjunction formula \cite[Formula II.8.20]{H} shows that $\mathcal{O}_{\Theta_{B}}(\Theta_{B})$ is the canonical bundle of $\Theta_{B}$. Since $\Theta_{B}$ has genus two, it follows that  $\gamma$ is a double cover of $\mathbb{P}^1$. We can identify the fibers explicitly: indeed we see from the explicit form (\ref{gaussmap})  that the Gauss map  is invariant under the involution $u\mapsto -u$ of $\Theta_{B}$. Since this map has degree precisely two by the previous discussion, it follows that the fibers of $\gamma$ consist precisely of opposite points $u,-u$, with $u\in \Theta_{B}$.
\end{remark}

Now set $d\geq 2$. Since $\Theta_{B}$ is smooth, Corollary \ref{globalgeneration} shows that the map 
\begin{equation}
\phi_{B,d}\colon A_{B} \to \mathbb{P}^{N_{2,d}}, \qquad N_{2,d} = \frac{(d+2)(d+1)}{2}-3
\end{equation}
is everywhere defined. Moreover, since $\Theta_{B}$ is irreducible, Remark \ref{projection} shows that this map can be realized as a composition
\begin{equation}
A_{B} \overset{\psi_{d,B}}{\longrightarrow} \mathbb{P}^{d^2-1} \overset{\pi}{\dashrightarrow} \mathbb{P}^{N_{2,d}}
\end{equation}
where $\pi$ is a linear projection. We study this situation when $d=2$ and $d=3$.
\begin{itemize}
\item $d=2$. In this case, we see that $d^2-1=N_{2,d}=3$, so that the projection $\pi$ must be the identity. Hence we can identify the map $\phi_{2,B}$ with $\psi_{2,B}$. This map is well studied: it induces a degree 2 cover of $A_{B}$ onto a quartic surface $X_B \subseteq \mathbb{P}^3$ with $16$ nodes. This surface is called the \textit{Kummer surface} of $A_{B}$, and the map $\phi_{2,B}\colon A_{B} \to X_{B}$ realizes it as the quotient of $A_{B}$ under the involution $u\mapsto -u$. In particular, the 16 nodes correspond to the points of order two in $A_{B}$.

\item $d=3$. In this case, we see that $N_{2,d}=7$, whereas $d^2-1 = 8$. Hence,  the map $\phi_{3,B}$ is the composition of the closed embedding $\psi_{3,B}\colon A_{B} \hookrightarrow \mathbb{P}^8$, with a projection $\pi\colon \mathbb{P}^8\dashrightarrow \mathbb{P}^7$  from a point $P$. We want to identify geometrically this point. To do this, take any $u\in \Theta_{B}$ which is not a point of order two: then we claim that $\phi_{3,B}(u)=\phi_{3,B}(-u)$. Taking this claim for granted, we can rephrase it by saying that the two points $\psi_{3,B}(u),\psi_{3,B}(-u)$ must lie on a common line with the point $P$. Hence, the point $P$ can be characterized as the common intersection point of all the lines in $\mathbb{P}^8$ spanned by the couples $\psi_{3,B}(u),\psi_{3,B}(-u)$, for $u\in \Theta_{B}$.

  To conclude we need to prove the claim. We know that the restriction of $\phi_{3,B}$ to the theta divisor $\Theta_{B}$ is the composition of the Gauss map $\gamma\colon \Theta_{B} \to \mathbb{P}^1$ with a closed embedding. Hence, it is enough to show that $\gamma(u)=\gamma(-u)$ for all $u\in \Theta_{B}$. However, this follows from the discussion in Remark \ref{gaussmapabsurf}. 
\end{itemize}

What if $\Theta_{B}$ is reducible? Then \cite[Corollary 11.8.2]{BL} shows that the polarized abelian surface $(A_{B},\Theta_{B})$ must be a product of elliptic curves. If we want to study the geometry of the maps $\phi_{d,B}$, we can use Proposition \ref{statisticgeometry}, and  assume that $B = \left(\begin{smallmatrix} B_1 & 0 \\ 0 & B_2 \end{smallmatrix}\right)$, so that $(A_{B},\Theta_{B}) = (A_{B_1},\Theta_{B_1}) \times (A_{B_2},\Theta_{B_2})$. By Proposition \ref{independence}, this corresponds to the case of two independent discrete Gaussians, and in particular it is straightforward to compute the cumulants $\kappa_{(a_1,a_2)}$.
We have that
$\kappa_{(a_1,0)}((u_1,u_2),B) = \kappa_{a_1}(u_1,B_1), \kappa_{(0,a_2)}((u_1,u_2),B) = \kappa_{a_2}(u_1,B_1)$, and the mixed cumulants $\kappa_{(a_1,a_2)}((u_1,u_2),B)$ vanish whenever one of $a_1,a_2$ is nonzero.
Hence,  we can write the rational maps $\phi_{d,B}$ as the composition of the two maps
\begin{equation}\label{mapreducibletheta}
A_{B_1}\times A_{B_2} \overset{\phi_{d,B_1}\times \phi_{d,B_2}}{\longrightarrow} \mathbb{P}^{d-1}\times \mathbb{P}^{d-1} \overset{f_d}{\dashrightarrow} \mathbb{P}^{2d-2},
\end{equation}
where the second map $f_d$ is given by
\begin{equation}
([X_0,\dots,X_{d-1}],[Y_0,\dots,Y_{d-1}]) \mapsto [X_0Y_0,Y_0X_1,\dots,Y_0X_{d-1},X_0Y_1,\dots,X_0Y_{d-1}]   
\end{equation}

\begin{remark}\label{mapreduciblethetacomplement}
In particular, we observe that the map $\phi_{d,B_1}\times \phi_{d,B_2}$ sends the open subset $A_{B}\setminus \Theta_{B}$ into the open subset $\{ X_0 \ne 0 \} \times \{ Y_0\ne 0 \}$, and the map $f_d$ restricted to this subset is an isomorphism onto the image. 
\end{remark}

As in Theorem \ref{thirdmomentellcurves}, we can use the statistical map $\phi_{3,B}$ to see that discrete Gaussians of dimension $2$ with the same parameter $B$ are determined by the moments up to order three. 

\begin{theorem}\label{thirdmomentabsurf}
Fix $B \in \mathbb{H}_2$ and let $X_{(u,B)},X_{(u',B)}$ be two discrete Gaussians with the same parameter $B$. If all the moments up to order three coincide, then they have the same distribution.
\end{theorem}
\begin{proof}
  Proceeding in the same way as in dimension $1$, it is enough to show that the map $\phi_{3,B}\colon A_{B} \to \mathbb{P}^3$ is injective on the open set $A_{B}\setminus \Theta_{B}$.

  Suppose first that $\Theta_{B}$ is irreducible. If we look at $A_{B}$ as embedded in $\mathbb{P}^8$ by $\psi_{3,B}$, the injectivity of $\phi_{3,B}$ is equivalent to saying that no two points $x,y$ in $A_{B}\setminus \Theta_{B}$ lie on a line passing through the point $P$. Suppose that this happens and take two distinct points $u,-u \in \Theta_{B}$ exchanged by the involution. Then the line that they span is also passing through $P$, so that the four distinct points $x,y,u,-u$ span a two-dimensional plane in $\mathbb{P}^8$. However, Lemma \ref{lemmaplane} below shows that in this case $x,y$ belong to $\Theta_{B}$ as well, which is absurd.

Suppose instead that $\Theta_{B}$ is reducible: then as before we can assume $B = \left(\begin{smallmatrix} B_1 & 0 \\ 0 & B_2 \end{smallmatrix}\right)$, so that  
 the map $\phi_{3,B}$ is a composition of the map $\phi_{3,B_1}\times \phi_{3,B_2}$ and of the map $f_3$,  as in (\ref{mapreducibletheta}). Since $\phi_{3,B_1}\times \phi_{3,B_2}$ is a product of closed embeddings, it is injective. Moreover,  Remark \ref{mapreduciblethetacomplement} shows that the map $f_3$ is injective when restricted to the image of $A_{B}\setminus \Theta_{B}$, and we are done.  
\end{proof}

To conclude the previous proof, we need a more technical lemma from algebraic geometry. In particular, we will make use some intersection theory on surfaces, for which we refer to \cite[Section V.1]{H}.

\begin{lemma}\label{lemmaplane}
  Fix $B\in\mathbb{H}_2$ such that $\Theta_{B}$ is irreducible and consider the embedding $\psi_{3,B}\colon A_{B} \hookrightarrow \mathbb{P}^8$. Take two opposite points $u,-u\in \Theta_{B}$ and other two points $x,y\in A_{B}$.
Suppose that the points $\psi_{3,B}(u),\psi_{3,B}(-u),\psi_{3,B}(x),\psi_{3,B}(y)$ span a plane of dimension two in $\mathbb{P}^8$. Then $x,y$ belong to $\Theta_{B}$ as well.
\end{lemma}
\begin{proof}
With these assumptions, the proof of \cite[Theorem 5.7]{TER} shows that the points $u,-u,x,y$ lie on a smooth curve $C\subseteq A_{B}$ of genus two, such that $(C\cdot \Theta_{B})=2$. Since $(C^2)=2$, an application of the Hodge Index Theorem \cite[Theorem V.1.9]{H} shows that $C$ and $\Theta_{B}$ are actually numerically equivalent curves on $A_{B}$. Since $A_{B}$ is an abelian variety,  it follows \cite[Theorem 4.11.1]{BL} that $C$ is actually the translate of $\Theta_{B}$ by a point $c\in A_{B}$, so that $C=\Theta_{B}+c$. If $c=0$, we are done. Suppose that $c\ne 0$: this means \cite[Lemma 11.3.1]{BL} that the restriction $\mathcal{O}_{\Theta_{B}}(C)$ is a line bundle of degree two which is not the canonical line bundle on the curve $\Theta_{B}$. However, the line bundle $\mathcal{O}_{\Theta_{B}}(C)$ corresponds by definition to the divisor $\Theta_{B}\cap C = \{u,-u\}$, and we have seen in Remark \ref{gaussmapabsurf} that this is a canonical divisor on $\Theta_{B}$. This gives a contradiction and we are done.
\end{proof}

\subsection{Open questions}
From these results in dimension one and two, there are natural questions that arise:

\begin{question}
  We have seen in Theorem \ref{thirdmomentellcurves} and Theorem \ref{thirdmomentabsurf} that two Gaussian distributions in dimension $1$ or $2$ with the same parameter $B$ can be distinguished by the moments up to order three. Is this true for every dimension $g$?
\end{question}

 Following the proofs of Theorems \ref{thirdmomentellcurves} and  \ref{thirdmomentabsurf}, a geometric way to prove this would be to show that the maps
  \begin{equation}
  \phi_{3,B}\colon A_{B}\setminus \Theta_{B} \longrightarrow \mathbb{P}^{N_{3,g}}
\end{equation}
are injective for every $B\in\mathbb{H}_g$.
Moreover, it could be that for a fixed parameter $B$, two Gaussian distributions can be distinguished just by the moments up to order two. We would like to identify when this happens too.

We can generalize this question to all Gaussian distributions at the same time:

\begin{question}
  We have seen in Theorem \ref{thirdmomentellcurves} that in general mean and covariance are not enough to determine a discrete Gaussian. Could this be true if we take all moments up to order three? Or up to an higher order $d$?
  Geometrically, this would mean to study the fibers of the map
  \begin{equation}
   \Phi_{d,g}\colon \mathbb{C}^g\times \mathbb{H}_g \setminus \Theta_g \longrightarrow  \mathbb{P}^{N_{d,g}}, \qquad (u,B) \mapsto \phi_{d,B}(u).
  \end{equation}
\end{question}

\begin{remark}
Taking our correspondence further, mixtures of discrete Gaussians correspond to secants of abelian varieties. More precisely, if $B$ is fixed then mixtures of two discrete Gaussians with parameters $(u_1,B), (u_2,B)$ form a secant line in the corresponding abelian variety, while mixtures with $(u_1,B_1), (u_2,B_2)$ lie in a secant line to the universal family $\mathcal{U}_g$. This should be a very interesting connection to explore, with natural recurring questions such as identifiability from moments (the case of continuous Gaussian mixtures is treated in \cite{IDENT}).
\end{remark}

\section*{Acknowledgements}
We would like to give special thanks to Bernd Sturmfels for dreaming and suggesting the connection between multivariate discrete Gaussians and theta functions. We thank Lynn Chua for her help with SAGE computations. Thanks to Samuel Grushevsky for his interest and his valuable comments. Thanks also to Daniel Dadush, Oded Regev and Noah Stephens for offering the CS perspective of discrete Gaussians, and to Robert Gray for his interest in this work. We also thank the anonymous referees who provided constructive feedback. Daniele Agostini was supported by the DAAD, the DFG Graduiertenkolleg 1800, the DFG Schwerpunkt 1489 and the Berlin Mathematical School. He also thanks the Stony Brook Mathematics Department for its hospitality while writing this work. Carlos Am\'endola was supported by the Einstein Foundation Berlin.

\end{document}